\documentclass[a4paper,12pt]{article}	
\usepackage{amsmath}
\usepackage{amssymb}	
\usepackage{amsthm}\usepackage{latexsym}
\usepackage{amsmath}
\usepackage{amsthm}
\usepackage{graphicx}
\usepackage{txfonts}
\usepackage{bm}
\usepackage{color}
\usepackage{epic,eepic}

\usepackage{geometry}
\numberwithin{equation}{section}

\newtheorem{theorem}{Theorem}[section]
\newtheorem{proposition}[theorem]{Proposition}
\newtheorem{propositionM}[theorem]{Proposition}

\newtheorem{lemmaM}[theorem]{Lemma}
\newtheorem{remark}{Remark}[section]
\newtheorem{example}{Example}[section]

\newcommand{\OMIT}[1]{{\bf [OMIT:} #1 \ {\bf --- end OMIT] }}  
   \renewcommand{\OMIT}[1]{}            

\newcommand{\RR}{{\mathbb{R}}}
\newcommand{\ZZ}{{\mathbb{Z}}}

\newcommand{\dom}{{\rm dom\,}}

\newcommand{\Rminf}{\RR \cup \{ -\infty \}}

\newcommand{\finbox}{\hspace*{\fill}$\rule{0.2cm}{0.2cm}$}

\newcommand{\Mnat}{{M$^{\natural}$}}

\newcommand{\BvexS}{\mbox{\rm (B-EXC[$\mathbb{B}$])} }
\newcommand{\BvexSb}{\mbox{\rm\bf (B-EXC[$\mathbb{B}$])}}
\newcommand{\BnvexS}{\mbox{\rm (B$\sp{\natural}$-EXC[$\mathbb{B}$])} }
\newcommand{\BnvexSb}{\mbox{\rm\bf (B$\sp{\natural}$-EXC[$\mathbb{B}$])}}

\newcommand{\BvexwS}{\mbox{\rm (B-EXC$_{\rm w}$[$\mathbb{B}$])}}
\newcommand{\BvexwSb}{\mbox{\rm\bf (B-EXC$_{\rm\bf w}$[$\mathbb{B}$])}}

\newcommand{\BvexmSb}{\mbox{\rm\bf (B-EXC$_{\rm\bf m}$[$\mathbb{B}$])}}

\newcommand{\BnvexmSb}{\mbox{\rm\bf (B$\sp{\natural}$-EXC$_{\rm\bf m}$[$\mathbb{B}$])}}

\newcommand{\McavS}{\mbox{\rm (M-$\overline{{\rm EXC}}$[$\mathbb{B}$])} }
\newcommand{\McavSb}{\mbox{\rm\bf (M-$\overline{{\rm\bf EXC}}$[$\mathbb{B}$])}}

\newcommand{\McavwS}{\mbox{\rm (M-$\overline{{\rm EXC}}_{\rm w}$[$\mathbb{B}$])} }
\newcommand{\McavwSb}{\mbox{\rm\bf (M-$\overline{{\rm\bf EXC}}_{\rm\bf w}$[$\mathbb{B}$])} }
\newcommand{\MncavS}{\mbox{\rm (M$\sp{\natural}$-$\overline{{\rm EXC}}$[$\mathbb{B}$])} }
\newcommand{\MncavSb}{\mbox{\rm\bf (M$\sp{\natural}$-$\overline{{\rm\bf EXC}}$[$\mathbb{B}$])}}

\newcommand{\MncavlocS}{\mbox{\rm (M$\sp{\natural}$-$\overline{{\rm EXC}}_{\rm loc}$[$\mathbb{B}$])} }
\newcommand{\MncavlocSb}{\mbox{\rm\bf (M$\sp{\natural}$-$\overline{{\rm\bf EXC}}_{\rm\bf loc}$[$\mathbb{B}$])} }

\newcommand{\McavlocS}{\mbox{\rm (M-$\overline{{\rm EXC}}_{\rm loc}$[$\mathbb{B}$])} }
\newcommand{\McavlocSb}{\mbox{\rm\bf (M-$\overline{{\rm\bf EXC}}_{\rm\bf loc}$[$\mathbb{B}$])} }

\newcommand{\McavmS}{\mbox{\rm (M-$\overline{{\rm EXC}}_{\rm m}$[$\mathbb{B}$])} }
\newcommand{\MncavmS}{\mbox{\rm (M$\sp{\natural}$-$\overline{{\rm EXC}}_{\rm m}$[$\mathbb{B}$])} }
\newcommand{\MncavmsS}{\mbox{\rm (M$\sp{\natural}$-$\overline{{\rm EXC}}_{\rm ms}$[$\mathbb{B}$])} }

\newcommand{\McavmSb}{\mbox{\rm\bf (M-$\overline{{\rm\bf EXC}}_{\rm\bf m}$[$\mathbb{B}$])}}
\newcommand{\MncavmSb}{\mbox{\rm\bf (M$\sp{\natural}$-$\overline{{\rm\bf EXC}}_{\rm\bf m}$[$\mathbb{B}$])}}
\newcommand{\MncavmsSb}{\mbox{\rm\bf (M$\sp{\natural}$-$\overline{{\rm\bf EXC}}_{\rm\bf ms}$[$\mathbb{B}$])}}

\begin{document}

\title{Exchange Properties of M$\sp{\natural}$-concave Set Functions 
\\
and Valuated Matroids
}

\author{
Kazuo Murota%
\thanks{
The Institute of Statistical Mathematics,
Tokyo 190-8562, Japan; and
Tokyo Metropolitan University, 
Tokyo 192-0397, Japan}
}

\date{May 2021}

\maketitle

\begin{abstract}
This is a survey article on the exchange properties characterizing 
\Mnat-concave set functions and valuated matroids (M-concave set functions).
The objective of this paper is to 
collect related results scattered in the literature
and to give (reasonably) self-contained elementary proofs for them.
\end{abstract}

\Mnat-concave set function,
Valuated matroid,
Exchange properties.



\tableofcontents


\section{Introduction}
\label{SCintro}

In discrete convex analysis
\cite{Mdca98=valmat, Mdcasiam=valmat},
\cite[Chapter VII]{Fuj05book=valmat},
M-concave functions and their variant called M$\sp{\natural}$-concave functions 
play a major role.
The concepts of M-concave functions 
and M$\sp{\natural}$-concave functions 
were introduced, respectively, by
Murota \cite{Mstein96=valmat}
and Murota--Shioura \cite{MS99gp=valmat}
for functions defined on the integer vectors.
In this paper we deal with
\Mnat- and M-concave set functions, that is,
\Mnat- and M-concave functions defined on $\{ 0, 1 \}$-vectors.
M-concave set functions are exactly the same as
valuated matroids introduced earlier by 
Dress--Wenzel \cite{DW90=valmat,DW92=valmat}.
\Mnat-concavity of a set function has significance in economics,
as it is equivalent to the gross substitutes property (GS)
of Kelso--Crawford \cite{KC82=valmat}.
See Murota \cite[Chapter 11]{Mdcasiam=valmat}, Murota \cite{Mdcaeco16=valmat},
and Shioura--Tamura \cite{ST15jorsj=valmat}
for more about economic significance of \Mnat-concavity.


In this paper we are interested in various types of 
exchange properties characterizing 
\Mnat-concave and M-concave set functions. 
We aim at collecting related results scattered in the literature
and giving (reasonably) self-contained elementary proofs for them.
The exchange properties for \Mnat-concave set functions
treated in Section~\ref{SCmncavsetfn} 
are mostly based on 
Murota--Shioura \cite{MS99gp=valmat,MS18mnataxiom=valmat}.
The proofs given in Section~\ref{SCmncavexcprf} are obtained 
by translating the proofs given in \cite{MS18mnataxiom=valmat} 
for functions on the integer lattice
to those for set functions (with some simplifications). 
The exchange properties for M-concave set functions
treated in Section~\ref{SCmcavsetfn} are
based on Murota \cite{Mmax97=valmat,Mspr2000=valmat,Mdcasiam=valmat},
while the proofs are made consistent with those in Section~\ref{SCmncavexcprf}.
Multiple exchange properties 
treated in Section~\ref{SCexchange01mult} 
are taken from Murota \cite{Mmultexc18=valmat,Mmultexcstr18=valmat}.


\section{\Mnat-concave Set Functions}
\label{SCmncavsetfn}


\subsection{Definition}
\label{SCmncavsetfnDef}

Let 
$f: 2\sp{N} \to \Rminf$ be
a real-valued set function on $N = \{ 1,2,\ldots, n \}$ and 
$\mathcal{F} = \dom f$ be the {\em effective domain} of $f$
defined by
\begin{equation} \label{effdom01def}
\dom f = \{ X \subseteq N \mid f(X) > -\infty \}.
\end{equation}
We always assume that $\dom f$ is nonempty.

We say that $f$
is an {\em \Mnat-concave function}, if,
for any $X, Y \in \mathcal{F}$ and $i \in X \setminus Y$,
we have (i)
$X - i \in \mathcal{F}$, $ Y + i \in \mathcal{F}$ and
\begin{equation}  \label{mnatcav1}
f( X) + f( Y ) \leq f( X - i ) + f( Y + i ),
\end{equation}
or (ii) there exists some $j \in Y \setminus X$ such that
$X - i +j \in \mathcal{F}$, $ Y + i -j  \in \mathcal{F}$ and
\begin{equation}  \label{mnatcav2}
f( X) + f( Y ) \leq  f( X - i + j) + f( Y + i -j).
\end{equation}
Such property is referred to as an {\em exchange property}.
Here we use short-hand notations
$X - i = X \setminus  \{ i \}$ and $Y + i = Y \cup \{ i \}$
as well as  $X - i + j =(X \setminus  \{ i \}) \cup \{ j \}$
and $Y + i - j =(Y \cup \{ i \}) \setminus \{ j \}$.

An \Mnat-concave function can also be defined 
without explicit reference to its effective domain
by the following expression of the exchange property:%
\footnote{
In acronym $\MncavS$, 
``EXC''stands for ``exchange'' 
and $\overline{\rm \phantom{EXC}}$ in ``$\overline{\rm EXC}$'' 
indicates concavity (in contrast to convexity);
``$\mathbb{B}$'' stands for ``binary'' showing that this condition
applies to set functions
(in contrast to ``$\mathbb{Z}$'' for functions on the integer lattice).
}  
\begin{description}
\item[\MncavSb] 
For any $X, Y \subseteq N$ and $i \in X \setminus Y$, we have
\begin{align}
f( X) + f( Y )   &\leq 
   \max\left( f( X - i ) + f( Y + i ), \ 
 \max_{j \in Y \setminus X}  \{ f( X - i + j) + f( Y + i -j) \}
       \right) ,
\label{mnatconcavexc2}
\end{align}
\end{description}
where $(-\infty) + a = a + (-\infty) = (-\infty) + (-\infty)  = -\infty$ for $a \in \RR$,
$-\infty \leq -\infty$, and
the maximum taken over an empty set is defined to be $-\infty$.


The effective domain of an \Mnat-concave function
is equipped with a nice combinatorial structure.
Let $\mathcal{F}$ be the effective domain of an \Mnat-concave function $f$.
As a consequence of the exchange property $\MncavS$ of function $f$,
the set family $\mathcal{F}$
satisfies the following exchange property:
\begin{description}
\item[\BnvexSb] 
For any $X, Y \in \mathcal{F}$ and $i \in X \setminus Y$, \ 
we have
(i) $X - i \in \mathcal{F}$, $ Y + i \in \mathcal{F}$
\ or \  
(ii) there exists some $j \in Y \setminus X$ such that
$X - i +j \in \mathcal{F}$, $ Y + i -j  \in \mathcal{F}$.
\end{description}
This means that $\mathcal{F}$ forms a matroid-like structure, 
called a {\em generalized matroid} ({\em g-matroid}).
In this paper we refer to it as an {\em \Mnat-convex family}
to emphasize its role for discrete convexity.

An \Mnat-convex family containing the empty set
as its member
is exactly the family of independent sets of a matroid.
An \Mnat-convex family 
consisting of equi-cardinal sets forms 
the family of bases of a matroid, 
which may also be called an {\em M-convex family}\index{M-convex family}.
More generally, an equi-cardinal subfamily%
\footnote{
An equi-cardinal subfamily of $\mathcal{F}$ means
a set family represented as
$\{ X \in \mathcal{F} \mid |X| = r \}$ for some $r \in \ZZ$.
} 
of an \Mnat-convex family is an M-convex family.
M-convex families (matroid bases) and
\Mnat-convex families (g-matroids) 
are fully studied in the literature of matroids
(Frank \cite{Fra11book=valmat}, 
Oxley \cite{Oxl11=valmat},
Schrijver \cite{Sch03=valmat}, and
Welsh \cite{Wel76=valmat}).

\subsection{Exchange properties characterizing \Mnat-concave functions}
\label{SCexchange01glob}

Under the assumption that the 
$\dom f$ contains the empty set,
the \Mnat-concavity can be characterized by a simpler condition:
\begin{description}
\item[{\bf (P1$[\mathbb{B}]$)}] 
For any $X, Y \subseteq N$ with $|X| < |Y|$,
we have
\begin{align}
f( X) + f( Y )   &\leq 
 \max_{j \in Y \setminus X}  \{ f( X  + j) + f( Y  -j) \} .
\label{mnatP1=01}
\end{align}
\end{description}
This condition (P1$[\mathbb{B}]$), 
applicable to a pair of subsets $(X, Y)$ with with $|X| < |Y|$, 
makes the pair closer with an appropriate element $j \in Y \setminus X$
without decreasing the sum of the function values.

\begin{theorem}\label{THmnatcavP1hered01}
Let $f: 2\sp{N} \to \Rminf$ 
be a set function 
with $\dom f$ containing the empty set.
Then 
$f$ is \Mnat-concave
if and only if it satisfies
{\rm (P1$[\mathbb{B}]$)}.
\end{theorem}
\begin{proof}
The proof is given in Section~\ref{SCproofmnatcavP1hered01}.
\end{proof}

\Mnat-concave functions satisfy
other cardinality-restricted exchange properties
 and are characterized by some combinations thereof.
The exchange properties 
(P2$[\mathbb{B}]$) and (P3$[\mathbb{B}]$) below
exclude the first possibility (\ref{mnatcav1}) in $\MncavS$
when $|X| \leq |Y|$,
and (P4$[\mathbb{B}]$) is a special case of $\MncavS$ with $|X| > |Y|$.

\begin{description}

\item[{\bf (P2$[\mathbb{B}]$)}] 
For any $X, Y \subseteq N$ with $|X| = |Y|$ and $i \in X \setminus Y$,
we have
\begin{align}
f( X) + f( Y )   &\leq 
 \max_{j \in Y \setminus X}  \{ f( X - i + j) + f( Y + i -j) \} ;
\label{mnatP2=01}
\end{align}

\item[{\bf (P3$[\mathbb{B}]$)}] 
For any $X, Y \subseteq N$ with $|X| < |Y|$ and $i \in X \setminus Y$,
we have
\eqref{mnatP2=01}:
\begin{align}
f( X) + f( Y )   &\leq 
 \max_{j \in Y \setminus X}  \{ f( X - i + j) + f( Y + i -j) \};
\label{mnatP3=01}
\end{align}

\item[{\bf (P4$[\mathbb{B}]$)}] 
For any $X, Y \subseteq N$ with $|X| > |Y|$ and $i \in X \setminus Y$,
we have
{\rm (\ref{mnatconcavexc2})}:
\begin{align}
f( X) + f( Y )   &\leq 
   \max\left( f( X - i ) + f( Y + i ), \ 
 \max_{j \in Y \setminus X}  \{ f( X - i + j) + f( Y + i -j) \}
       \right) .
\label{mnatP4=01}
\end{align}
\end{description}

The following theorem gives two characterizations of \Mnat-concave set functions.

\begin{theorem}\label{THmconcavcardexc01}
Let $f: 2\sp{N} \to \Rminf$ be a set function with $\dom f \not= \emptyset$.

\noindent
{\rm (1)}
$f$ is \Mnat-concave
if and only if it satisfies
{\rm (P1$[\mathbb{B}]$)} and {\rm (P2$[\mathbb{B}]$)}.

\noindent
{\rm (2)}
$f$ is \Mnat-concave
if and only if it satisfies
{\rm (P2$[\mathbb{B}]$)}, {\rm (P3$[\mathbb{B}]$)}, and {\rm (P4$[\mathbb{B}]$)}.
\end{theorem}
\begin{proof}
The proof is given in Section~\ref{SCproofmnatexccardloc}.
\end{proof}

\begin{remark} \rm  \label{RMmconcavcardexc01bib}
Theorem~\ref{THmnatcavP1hered01}
and Theorem~\ref{THmconcavcardexc01}(1) are explicit in 
Murota--Shioura \cite{MS18mnataxiom=valmat}
as Corollary 1.4 and Corollary 1.3, respectively.
Theorem~\ref{THmconcavcardexc01}(2) is an adaptation 
of Theorem 2.1 in \cite{MS18mnataxiom=valmat} to set functions. 
\finbox
\end{remark}

Other types of exchange properties for \Mnat-concavity
shall be treated later,
local exchange properties in Section~\ref{SCexchange01loc}
and multiple exchange properties in Section~\ref{SCexchange01mult}.


\subsection{Local exchange properties characterizing \Mnat-concave functions}
\label{SCexchange01loc}



\Mnat-concavity can be characterized by local exchange properties.
The conditions 
(L1$[\mathbb{B}]$)--(L3$[\mathbb{B}]$) 
below are indeed ``local'' in the sense that they require 
the exchangeability of the form (\ref{mnatconcavexc2}) only for 
$(X,Y)$ with $\max(| X \setminus Y | , | Y \setminus X |) \leq 2$.

\begin{description}
\item[{\bf (L1$[\mathbb{B}]$)}] 
For any $Z \subseteq N$ and distinct $i,j \in N \setminus Z$,
we have
\begin{align}
f( Z + i + j ) + f( Z ) \leq f(Z + i) + f(Z + j) ;
\label{mnatconcavexc20loc}
\end{align}

\item[{\bf (L2$[\mathbb{B}]$)}] 
For any $Z \subseteq N$ and distinct $i,j,k \in N \setminus Z$,
we have
\begin{align}
& f( Z + i + j ) + f( Z + k)  
\notag \\  & 
\quad \leq 
\max\left[ f(Z + i + k) + f(Z + j), \  f(Z + j + k) + f(Z + i) \right] ;
\label{mnatconcavexc21loc}
\end{align}

\item[{\bf (L3$[\mathbb{B}]$)}] 
For any $Z \subseteq N$ and distinct $i,j,k,l \in N \setminus Z$,
we have
\begin{align}
& 
 f( Z + i + j ) + f( Z + k + l)  
\notag \\  & 
\quad \leq 
\max\left[  f(Z + i + k) + f(Z + j +l ), \   f(Z + j + k) + f(Z + i + l) \right].  
\label{mnatconcavexc22loc}
\end{align}
\end{description}

\begin{remark} \rm  \label{RMmlocnonunimax}
Condition (L2$[\mathbb{B}]$)
is equivalent to saying that,
for any $Z \subseteq N$
 and distinct $i,j,k \not\in Z$, the maximum value in 
$\{ 
f( Z + i + j ) + f( Z + k), \
f(Z + i + k) + f(Z + j), \ f(Z + j + k) + f(Z + i) \}$
is attained by at least two elements therein.
Similarly,  (L3$[\mathbb{B}]$)
is equivalent to saying that,
for any $Z \subseteq N$ 
and distinct $i,j,k,l \not\in Z$, the maximum value in 
$\{  f( Z + i + j ) + f( Z + k + l), \ 
 f(Z + i + k) + f(Z + j +l ), \   f(Z + j + k) + f(Z + i + l) \}$
is attained by at least two elements therein.
\finbox
\end{remark}

The set of the above three conditions will be referred to as $\MncavlocS$. 
That is, 
\begin{description}
\item[\MncavlocSb] 
The conditions (L1$[\mathbb{B}]$), (L2$[\mathbb{B}]$), and  (L3$[\mathbb{B}]$) hold.
\end{description}
\noindent
The following is a statement expected naturally.

\begin{propositionM} \label{PRmconcavlocexc01onlyif}
An \Mnat-concave set function satisfies $\MncavlocS$.
\end{propositionM}
\begin{proof}
The conditions
(L1$[\mathbb{B}]$) and (L2$[\mathbb{B}]$)
are immediate consequences of the exchange property $\MncavS$.
The derivation of (L3$[\mathbb{B}]$) below
demonstrates a typical reasoning about \Mnat-concavity.%
\footnote{
(L3$[\mathbb{B}]$) is a special case of (P2$[\mathbb{B}]$)
that appeared in Theorem~\ref{THmconcavcardexc01}.
However, we need to prove (L3$[\mathbb{B}]$) without using (P2$[\mathbb{B}]$),
since (P2$[\mathbb{B}]$) is not proved yet and moreover, the 
proof of (P2$[\mathbb{B}]$) given in Section~\ref{SCproofmnatexccardloc}
relies on this Proposition~\ref{PRmconcavlocexc01onlyif}.
} 

To simplify notations we 
write
$\alpha_{i} = f(Z + i)$,
$\alpha_{ij} = f(Z + i + j)$, and
$\alpha_{ijk} = f(Z + i + j + k)$, etc., 
and assume $i=1$, $j=2$, $k=3$, $l=4$ in 
(L3$[\mathbb{B}]$)
to obtain
\begin{equation} \label{mnatconcavexc22locAlpha}
 \alpha_{12}+\alpha_{34} \leq 
 \max\{\alpha_{13}+\alpha_{24},\alpha_{14}+\alpha_{23}\}.
\end{equation}
To prove this by contradiction, suppose that
\begin{equation}\label{vmloc22prf1}
\alpha_{12} + \alpha_{34} > 
       \max\{\alpha_{13} + \alpha_{24}, \alpha_{14} + \alpha_{23}\}.
\end{equation}
With the notation $A=\alpha_{12} + \alpha_{34}$
we obtain
\begin{align}
A = \alpha_{12} + \alpha_{34}
& \leq 
\max \{\alpha_{1}+\alpha_{234}, \alpha_{13} + \alpha_{24},
 \alpha_{14} + \alpha_{23}\} 
=\alpha_{1}+\alpha_{234}
\label{vmloc22prf2}
\end{align}
from  \MncavS (with $i=2$) and (\ref{vmloc22prf1}).
Similarly, we have
\begin{equation}\label{vmloc22prf3}
A \leq \alpha_{2}+\alpha_{134},
\qquad
A \leq \alpha_{3}+\alpha_{124},
\qquad
A \leq \alpha_{4}+\alpha_{123}.
\end{equation}
On the other hand, we have
\begin{align}
&\alpha_{1} + \alpha_{123} \leq \alpha_{12} + \alpha_{13},
\qquad 
\alpha_{2} + \alpha_{234} \leq \alpha_{23} + \alpha_{24}, 
\label{vmloc22prf42} 
\\
&\alpha_{3} + \alpha_{134} \leq \alpha_{13} + \alpha_{34},
\qquad 
\alpha_{4} + \alpha_{124} \leq \alpha_{14} + \alpha_{24}
\label{vmloc22prf44} 
\end{align}
by $\MncavS$.
By adding the four inequalities 
in (\ref{vmloc22prf2}) and (\ref{vmloc22prf3}) 
and using the inequalities 
in (\ref{vmloc22prf42}), (\ref{vmloc22prf44}), 
 and (\ref{vmloc22prf1}), we obtain
\begin{align*}
4A &\leq 
(\alpha_{1} + \alpha_{234}) 
+(\alpha_{2} + \alpha_{134}) 
+(\alpha_{3} + \alpha_{124}) 
+(\alpha_{4} + \alpha_{123}) 
\notag \\ & =
(\alpha_{1} + \alpha_{123}) 
+(\alpha_{2} + \alpha_{234}) 
+(\alpha_{3} + \alpha_{134}) 
+(\alpha_{4} + \alpha_{124}) 
\notag \\ & \leq
(\alpha_{12} + \alpha_{13}) 
+(\alpha_{23} + \alpha_{24}) 
+(\alpha_{13} + \alpha_{34}) 
+(\alpha_{14} + \alpha_{24}) 
\notag \\ & =
(\alpha_{12} + \alpha_{34}) 
+(\alpha_{23} + \alpha_{14}) 
+ 2(\alpha_{13} + \alpha_{24}) 
\notag \\ & 
< 4A .
\end{align*}
This is a contradiction. Thus (L3$[\mathbb{B}]$) is shown.
\end{proof}

The converse of Proposition~\ref{PRmconcavlocexc01onlyif} is also true,
that is, the local exchange property $\MncavlocS$ characterizes \Mnat-concavity 
under some assumption on the effective domain of the function.

\begin{theorem}\label{THmnatcavlocexc01}
A set function  $f: 2\sp{N} \to \Rminf$ 
is \Mnat-concave
if and only if the effective domain
$\dom f$ is an \Mnat-convex family and 
$\MncavlocS$  is satisfied.
\end{theorem}
\begin{proof}
The ``only if'' part is already shown in Proposition~\ref{PRmconcavlocexc01onlyif}.
The proof of the ``if'' part is given in Section~\ref{SCproofmnatexccardloc}.
\end{proof}

If the effective domain contains the empty set,
in addition to being an \Mnat-convex family,
we can dispense with the third condition (L3$[\mathbb{B}]$) in $\MncavlocS$,
as is stated in the following theorem.
Recall that an \Mnat-convex family containing the empty set is exactly 
the family of independent sets of a matroid.

\begin{theorem}\label{THmnatcavlocexc01hered}
Let $f: 2\sp{N} \to \Rminf$
be a set function such that $\dom f$ is 
the family of independent sets of a matroid
(an \Mnat-convex family containing the empty set).
Then $f$  is \Mnat-concave
if and only if it satisfies
{\rm (L1$[\mathbb{B}]$)} and {\rm (L2$[\mathbb{B}]$)}.
\end{theorem}
\begin{proof}
The proof is given in Section~\ref{SCproofmnatlocexc01hered}.
\end{proof}

\begin{remark} \rm  \label{RMmnatlocdomcond}
In Theorems \ref{THmnatcavlocexc01} and \ref{THmnatcavlocexc01hered}
the assumptions on $\dom f$ are indispensable.
For example, 
let $N=\{ 1,2,\ldots, 6 \}$
and define 
$f(\{ 1,2,3 \}) = f(\{ 4,5,6 \})=0$,
and $f(X)=-\infty$ for $X \not= \{ 1,2,3 \}, \{ 4,5,6 \}$.
This function $f$ is not \Mnat-concave, since 
$\dom f = \{ \{ 1,2,3 \}, \{ 4,5,6 \}  \}$ is not an \Mnat-convex family. 
However, $f$ satisfies the conditions 
(L1$[\mathbb{B}]$),
(L2$[\mathbb{B}]$), and (L3$[\mathbb{B}]$)
in a trivial manner, since the left-hand sides of 
(\ref{mnatconcavexc20loc})--(\ref{mnatconcavexc22loc})
are  always equal to $-\infty$.
\finbox
\end{remark}

\begin{remark} \rm  \label{RMmnatcavlocexc01bib}
Theorem~\ref{THmnatcavlocexc01} is due to Murota--Shioura \cite{MS18mnataxiom=valmat};
see also Murota \cite{Mstein96=valmat, Mdcasiam=valmat}, Murota--Shioura \cite{MS99gp=valmat}.
Theorem~\ref{THmnatcavlocexc01hered}
is due to 
Reijnierse--van Gallekom--Potters
\cite[Theorem~10]{RGP02=valmat}
(also M{\"u}ller \cite[Theorem~13.5]{Mul06=valmat},
Shioura--Tamura \cite[Theorem~6.5]{ST15jorsj=valmat}).
\finbox
\end{remark}

\begin{remark} \rm  \label{RMmnatlocweakcond}
In Theorem~\ref{THmnatcavlocexc01},
the assumption that $\dom f$ should be an \Mnat-convex family
can be weakened. 
See Proposition~\ref{PRmnatcavlocexcW01} in 
Section~\ref{SCproofmnatexccardloc}.
\finbox
\end{remark}


\section{Proofs about Exchange Properties of \Mnat-concave Functions}
\label{SCmncavexcprf}


Theorems about exchange properties 
stated in Section~\ref{SCmncavsetfn}  
are proved in this section.
The proofs are ordered in accordance with logical dependence, that is,
if the proof of Theorem~A relies on Theorem~B, 
Theorem~B is proved before it is used.

\subsection{Proof of Theorems \ref{THmconcavcardexc01}~and \ref{THmnatcavlocexc01}}
\label{SCproofmnatexccardloc}

In this section we prove the equivalence of the exchange properties
$\MncavS$, $\MncavlocS$, and some combinations of
(P1$[\mathbb{B}]$) to (P4$[\mathbb{B}]$)
stated in Theorems \ref{THmconcavcardexc01} and \ref{THmnatcavlocexc01}.
The proof is based on Murota--Shioura \cite{MS18mnataxiom=valmat}
and can be summarized as follows:
\begin{center}
\begin{tabular}{|ccc|}
\hline
\MncavS & 
{$\stackrel{\mbox{\small Prop.~\ref{PRmconcavlocexc01onlyif}}}{\Longrightarrow}$}
& $\MncavlocS$ $+$ DOM
 \\
 {\small Obviously}
 $\Uparrow$
\phantom{\small Obviously}
&&
\phantom{\small Lem \ref{LMmnatexcdiffcard01}, \ref{LMmnatexcequicard01}}
$\Downarrow$ 
{\small Lem \ref{LMmnatexcdiffcard01}, \ref{LMmnatexcequicard01}}
\\
(P2$[\mathbb{B}]$),
(P3$[\mathbb{B}]$),
(P4$[\mathbb{B}]$)
&
{$\stackrel{\mbox{\small Lem \ref{LMp12toP3}, \ref{LMp12toP4}}}{\Longleftarrow}$} 
& 
(P1$[\mathbb{B}]$),
(P2$[\mathbb{B}]$)
\\ \hline
\end{tabular}
\end{center}
where ``DOM'' denotes some conditions on the effective domain to be specfied below.
Recall that 
(P1$[\mathbb{B}]$) to (P4$[\mathbb{B}]$) 
are defined in Section~\ref{SCexchange01glob},
and  \MncavlocS 
consists of three local exchange properties
(L1$[\mathbb{B}]$), (L2$[\mathbb{B}]$), and (L3$[\mathbb{B}]$) 
in Section~\ref{SCexchange01loc}.
We use notation 
\[
f_{p}(X) = f[+p](X) =  f(X) + \sum_{i \in X} p_{i}
\]
for $p \in \mathbb{R}\sp{N}$ and $X \subseteq N$.


We start with some property of an \Mnat-convex family.

\begin{propositionM} \label{PRmnsetconnected01}
A set family $\mathcal{F}$ satisfying $\BnvexS$ 
has the following properties:%
\footnote{
(\ref{Fconnected<2}) is not implied by (\ref{Fconnected<1}) and (\ref{Fconnected=}).
For example,
$\mathcal{F} = \{ \{ 1,2 \}, \{ 1,4 \}, \{ 1,5 \}, \{ 4,5 \}$, $\{ 3,4,5 \} \}$
satisfies (\ref{Fconnected<1}) and (\ref{Fconnected=}), and not (\ref{Fconnected<2}).
} 
\begin{align} 
& \bullet
\mbox{
If $X, Y \in \mathcal{F}$ and $|X| < |Y|$, 
there exists $j \in Y \setminus X$ such that $Y - j \in \mathcal{F}$;
}
\label{Fconnected<1}
\\ &
\bullet
\mbox{
If $X, Y \in \mathcal{F}$, $|X| = |Y|$, and $X \not= Y$, then
}
\notag \\ & \qquad 
\mbox{
there exist $i \in X \setminus Y$ and $j \in Y \setminus X$
such that $Y + i - j  \in \mathcal{F}$;
}
\label{Fconnected=}
\\ &
\bullet
\mbox{
If $X, Y \in \mathcal{F}$, $|X| < |Y|$, and $X \setminus Y \not= \emptyset$, then
}
\notag \\ & \qquad 
\mbox{
there exist $i \in X \setminus Y$ and $j \in Y \setminus X$
such that $Y + i - j  \in \mathcal{F}$.
}
\label{Fconnected<2}
\end{align}
\end{propositionM}
\begin{proof}
We prove (\ref{Fconnected<1}) by induction on 
$|X \bigtriangleup Y| = |X \setminus Y| + |Y \setminus X|$.
 If 
$|X \bigtriangleup Y| = 1$,
then
$X = Y - j \in \mathcal{F}$ holds with the unique element $j \in Y \setminus X$.
 For the induction step, assume 
$|X \bigtriangleup Y| > 1$.
 If $X \setminus Y = \emptyset$, then
 \BnvexS  applied to $Y$, $X$, and an arbitrary $j \in Y \setminus X$
implies 
$Y - j \in \mathcal{F}$.
Otherwise,
take any $i \in X \setminus Y$.
 Then, we have
(i) $X - i \in \mathcal{F}$,  
 or  
(ii) there exists some $k \in Y \setminus X$ such that
$X - i +k \in \mathcal{F}$.
 Let $X' = X-i$ in case (i)
and $X' = X-i+ k$ in case (ii).
 Since $|X'| \le |X| < |Y|$
and 
$|X' \bigtriangleup Y| < |X \bigtriangleup Y|$,
we can apply the induction hypothesis to $X'$ and $Y$
to obtain $Y - j \in \mathcal{F}$ for some 
$j \in Y \setminus X' \subseteq Y  \setminus X$.

To prove (\ref{Fconnected=}) and (\ref{Fconnected<2})
assume $|X| \leq |Y|$ and $X \setminus Y \not= \emptyset$.
We apply
\BnvexS to $X$, $Y$, and $i \in X \setminus Y$,
to obtain
(i) $ Y + i \in \mathcal{F}$, 
 or 
(ii) there exists some $j \in Y \setminus X$ such that
$ Y + i -j  \in \mathcal{F}$.
In case (i) we have $|X| \le |Y| < |Y+i|$.
Then, by (\ref{Fconnected<1}),
we obtain
$Y+i - j \in \mathcal{F}$ for some
$j \in (Y+i) \setminus X = Y \setminus X$.
\end{proof}

To prove Theorem \ref{THmnatcavlocexc01},  we will show a stronger statement.

\begin{proposition}\label{PRmnatcavlocexcW01}
A set function  $f: 2\sp{N} \to \Rminf$ 
is \Mnat-concave
if and only if the effective domain
$\dom f$ satisfies
\eqref{Fconnected<1}, \eqref{Fconnected=}, and \eqref{Fconnected<2}
 and $\MncavlocS$  is satisfied.
\end{proposition}

In Lemmas \ref{LMmnatexcdiffcard01} and \ref{LMmnatexcequicard01} below, 
we derive (P1$[\mathbb{B}]$) and (P2$[\mathbb{B}]$) from $\MncavlocS$, respectively,
under the connectedness conditions
(\ref{Fconnected<1}), (\ref{Fconnected=}), and (\ref{Fconnected<2})
on $\dom f$.

\begin{lemmaM}  \label{LMmnatexcdiffcard01}
If $\dom f$ satisfies {\rm (\ref{Fconnected<1})} and {\rm (\ref{Fconnected<2})},
then $\MncavlocS$ implies {\rm (P1$[\mathbb{B}]$)}.
\end{lemmaM}

\begin{proof}
To prove (P1$[\mathbb{B}]$)
by contradiction, we assume that
there exists a pair $(X,Y)$ for which 
(\ref{mnatP1=01}) 
fails.
That is, we assume that the set of such pairs
\begin{align*}
 \mathcal{D} = \{(X, Y) \mid {} &  X, Y \in \dom f,\ |X| < |Y|, \ 
\\ &
f(X) + f(Y) > f(X + j) + f(Y - j)
\mbox{ for all } j \in Y \setminus X \}
\end{align*}
is nonempty.
Take a pair $(X,Y) \in \mathcal{D}$ with 
$|X \bigtriangleup Y| = |X \setminus Y| + |Y \setminus X|$
minimum.
For a fixed $\varepsilon > 0$, define 
$p \in \mathbb{R}\sp{N}$ by
\[
p_{j} = \left\{
\begin{array}{ll}
f(X) - f(X + j) & (j \in Y \setminus X,\ X + j \in \dom f),\\
f(Y - j) - f(Y) + \varepsilon
 & (j \in Y \setminus X,\ X + j \not\in \dom f, \ Y - j \in \dom f),\\
0 & (\mbox{otherwise}).
\end{array}
\right.
\]

\medskip

Claim 1:
\begin{eqnarray}
f_{p}(X + j) & = & f_{p}(X) 
   \qquad (j \in Y \setminus X,\ X + j \in \dom f),
\label{vmexcard1l1c11=01}\\
f_{p}(Y - j) & < & f_{p}(Y) \qquad (j \in Y \setminus X).
\label{vmexcard1l1c12=01}
\end{eqnarray}

\noindent (Proof of Claim~1) The equality
(\ref{vmexcard1l1c11=01}) is obvious from the definition of $p$.
If $X+j \in \dom f$,
(\ref{vmexcard1l1c12=01}) follows from (\ref{vmexcard1l1c11=01}) and
$f_{p}(X) + f_{p}(Y) > f_{p}(X + j) + f_{p}(Y - j)$.
If $X+j \not\in \dom f$,
(\ref{vmexcard1l1c12=01}) follows from the fact that
$f_{p}(Y - j) - f_{p}(Y) = -\varepsilon$ or $-\infty$
depending on whether $Y-j \in \dom f$ or not.

\medskip

In the following, we divide into cases,
Case~1:  $|Y|-|X| \ge 2$ and
Case~2:  $|Y|-|X| = 1$,
and derive a contradiction in each case.
We first consider Case~1.
 Since $|X| < |Y|$,
the assumption (\ref{Fconnected<1}) implies
$Y  - j_{0} \in \dom f$ for some $j_0 \in Y \setminus X$.

\medskip

Claim 2:  
For $Y' = Y   - j_{0}$ we have
\begin{equation}\label{vmexcard1clm3}
 f_{p}(X) + f_{p}(Y') > f_{p}(X + j) + f_{p}(Y' - j)
\qquad  (j \in Y' \setminus X).
\end{equation}

\noindent (Proof of Claim 2)
It suffices to consider the case where  $X + j \in \dom f$.
Then we have $f_{p}(X) = f_{p}(X + j)$ by (\ref{vmexcard1l1c11=01}).
 We also have
\begin{align*}
 f_{p}(Y' - j)  
& =
[  f_{p}(Y) + f_{p}(Y  - j_{0} - j) ] - f_{p}(Y)
\\
& \leq [ f_{p}(Y  - j_{0}) + f_{p}(Y - j) ]
 - f_{p}(Y)
\\
& < f_{p}(Y  - j_{0}) = f_{p}(Y') 
\end{align*}
by \MncavlocS and (\ref{vmexcard1l1c12=01}).
Therefore,
(\ref{vmexcard1clm3}) holds.

\medskip

Since  $|Y'| = |Y|-1 > |X|$, 
(\ref{vmexcard1clm3}) implies $(X, Y') \in \cal D$.
This contradicts the choice of $(X, Y)$, since
$|X \bigtriangleup Y'| = |X \bigtriangleup Y| -1 $.
Therefore, Case~1 cannot occur.
We next consider Case~2.

\medskip

Claim 3:
There exist $i_{0} \in X \setminus Y$ and $j_{0} \in Y \setminus X$
such that $Y + i_{0} - j_{0} \in \dom f$
and
\begin{equation}\label{vmexcard1l1a1-2}
f_{p}(Y + i_{0} - j_{0}) \geq  f_{p}(Y + i - j)
\qquad
(i \in X \setminus Y, \ j \in Y \setminus X).
\end{equation}

\noindent (Proof of Claim 3)
  We first note that $|X\setminus Y| \ge 1$ holds.
Indeed, if $|X\setminus Y|=0$, then
$X=Y-i$ and $Y=X+i$
for the unique element $i$ of $Y \setminus X$,
and $f(X) + f(Y) = f( X + i) + f( Y-i)$,
a contradiction to $(X,Y) \in \mathcal{D}$.
By assumption (\ref{Fconnected<2}), 
$Y + i - j \in \dom f$ for some $i \in X\setminus Y$ and $j \in Y \setminus X$.
Any pair $(i,j)$ that maximizes $f_{p}(Y + i - j)$
over $i \in X \setminus Y$ and $j \in Y \setminus X$
serves as $(i_{0}, j_{0})$.

\medskip

Claim 4:  
For $Y' = Y  + i_{0} - j_{0}$ we have (\ref{vmexcard1clm3}).

\noindent (Proof of Claim 4)
It suffices to consider the case where  $X + j \in \dom f$.
Then we have $f_{p}(X) = f_{p}(X + j)$ by (\ref{vmexcard1l1c11=01}).
 We also have
\begin{align*}
& f_{p}(Y' - j)  =
[ f_{p}(Y + i_{0} - j_{0} - j) +  f_{p}(Y)  ] - f_{p}(Y)
\\
& \leq  \max\{f_{p}(Y + i_{0} - j_{0}) + f_{p}(Y - j),
   f_{p}(Y + i_{0} - j) + f_{p}(Y - j_{0})\} - f_{p}(Y)
\\
& \leq 
 f_{p}(Y + i_{0} - j_{0}) +
 \max\{f_{p}(Y - j) - f_{p}(Y), f_{p}(Y - j_{0}) - f_{p}(Y)\}
\\
& < f_{p}(Y + i_{0} - j_{0}) = f_{p}(Y') 
\end{align*}
by $\MncavlocS$, (\ref{vmexcard1l1a1-2}), and (\ref{vmexcard1l1c12=01}).
Therefore,
(\ref{vmexcard1clm3}) holds.

\medskip

Since $|Y'| = |Y| > |X|$, 
(\ref{vmexcard1clm3})
implies $(X, Y') \in \cal D$.
This contradicts the choice of $(X, Y)$, since
$|X \bigtriangleup Y'| = |X \bigtriangleup Y| -2$.
Therefore, Case~2 cannot occur either. 
Hence, $\mathcal{D}$ must be empty, which means that
(P1$[\mathbb{B}]$) holds.
\end{proof}

\begin{lemmaM}  \label{LMmnatexcequicard01}
If $\dom f$ satisfies {\rm (\ref{Fconnected=})},
$\MncavlocS$ implies {\rm (P2$[\mathbb{B}]$)}.
\end{lemmaM}
\begin{proof}
To prove (P2$[\mathbb{B}]$)
by contradiction, we assume that
there exists a pair $(X,Y)$ for which 
(\ref{mnatP2=01}) 
 fails.
That is, we assume that the set of such pairs
\begin{align*}
 \mathcal{D} = \{(X, Y) \mid {} &  X, Y \in \dom f,\ 
|X| = |Y|,\ 
\exists i_{*} \in X \setminus Y \  \mbox{s.t.} 
\\ &
f(X) + f(Y)   > f(X- i_{*} +j ) + f(Y+ i_{*}-j )
\mbox{ for all } j \in Y \setminus X \}
\end{align*}
is nonempty.
Take a pair $(X,Y) \in \mathcal{D}$ with 
$|X \setminus Y|$ minimum,
and fix $i_{*} \in X \setminus Y$ 
appearing in the definition of $\mathcal{D}$.
We have
$|X \setminus Y| = |Y \setminus X| \geq 2$
by $\MncavlocS$.
For a fixed $\varepsilon > 0$, define 
$p \in \mathbb{R}\sp{N}$ by
\begin{equation} \label{mnatlocToP2pjdef}
\begin{array}{rcl}
p_{j} & = & \left\{
\begin{array}{ll}
f(X) - f(X- i_{*} +j) & (j \in Y \setminus X,\ 
         X- i_{*} +j \in \dom  f),\\
f(Y+ i_{*} -j) - f(Y) + \varepsilon \\
 & \hspace{-15mm} (j \in Y \setminus X,\ 
X- i_{*} +j \not\in \dom  f,\ 
Y+ i_{*} -j \in \dom  f),\\
0 & (\mbox{otherwise}).
\end{array}
\right.
\end{array}
\end{equation}

\medskip

Claim 1:
\begin{eqnarray}
f_{p}(X -  i_{*}  + j) & = & f_{p}(X) 
   \qquad (j \in Y \setminus X,\ X -  i_{*}  + j \in \dom  f),
         \label{vmexequiszl3-c1-1=01}\\
f_{p}(Y +  i_{*}  - j) 
       & < & f_{p}(Y) \qquad (j \in Y \setminus X).
\label{vmexequiszl3-c1-2=01}
\end{eqnarray}

\noindent (Proof of Claim~1) The equality
(\ref{vmexequiszl3-c1-1=01}) is obvious from the definition of $p$.
If $X-i_{*}+j \in \dom f$,
(\ref{vmexequiszl3-c1-2=01}) follows from (\ref{vmexequiszl3-c1-1=01}) and
$f_{p}(X) + f_{p}(Y) > f_{p}(X - i_{*} + j) + f_{p}(Y +i_{*} - j)$.
If $X-i_{*}+j \not\in \dom f$,
(\ref{vmexequiszl3-c1-2=01}) follows from the fact that
$f_{p}(Y +i_{*} - j) - f_{p}(Y) = -\varepsilon$ or $-\infty$
depending on whether $Y+i_{*}-j \in \dom f$ or not.

\medskip

Claim 2:
There exist $i_{0} \in (X \setminus Y) - i_{*}$ and $j_{0} \in Y \setminus X$
such that $Y + i_{0} - j_{0} \in \dom f$
and
\begin{equation}\label{vmexequiszl3-a1=01}
f_{p}(Y + i_{0} - j_{0}) \geq  f_{p}(Y + i_{0} - j)
\qquad
(j \in Y \setminus X).
\end{equation}

\noindent (Proof of Claim 2)
First, we show the existence of
$i_{0} \in X\setminus Y$ and $j \in Y\setminus X$
such that $Y + i_{0} - j \in \dom f$ and $i_{0} \not= i_{*}$.
By assumption (\ref{Fconnected=}),
there exist $i_{1} \in X \setminus Y$ and $j_{1} \in Y \setminus X$ such that
$Y'=Y + i_{1} - j_{1} \in \dom f$.
If $i_{1} \not= i_{*}$,  
we are done with $(i_{0},j) = (i_{1},j_{1})$. 
Otherwise, 
again by (\ref{Fconnected=}), 
there exist $i_{2} \in X \setminus Y'$ and $j_{2} \in Y' \setminus X$ such that
$Y'' = Y' + i_{2} - j_{2} \in \dom f$.
By the local exchange property 
(L3$[\mathbb{B}]$) in (\ref{mnatconcavexc22loc}) for $Y$ and $Y''$
we obtain
$Y + i_{2} - j_{1} \in \dom f$
or
$Y + i_{2} - j_{2} \in \dom f$.
Hence we can take $(i_{0},j) =(i_{2},j_{1})$ or $(i_{0},j) =(i_{2},j_{2})$,
where $i_{2}$ is distinct from $i_{*}$.
Next we choose the element $j_{0}$.
By the choice of $i_{0}$, we have $f_{p}(Y+i_{0}-j) > -\infty$ 
for some $j \in Y \setminus X$.
By letting $j_{0}$ to be an element
$j \in Y\setminus X$ that maximizes 
$f_{p}(Y+i_{0}-j)$, we obtain (\ref{vmexequiszl3-a1=01}).

\medskip

Claim 3:  
For $Y' = Y  + i_{0} - j_{0}$ we have
\begin{equation}\label{vmexequiszlem3-1=01}
 f_{p}(X) + f_{p}(Y') 
  > f_{p}(X -  i_{*}  + j) + f_{p}(Y' +  i_{*}  - j)
\qquad (j \in Y' \setminus X).
\end{equation}

\noindent (Proof of Claim 3)
It suffices to consider the case where $X - i_{*} + j \in \dom f$.
Then we have 
$f_{p}(X) = f_{p}(X - i_{*} + j)$ 
by (\ref{vmexequiszl3-c1-1=01}).
We also have
\begin{align*}
& f_{p}(Y'+ i_{*} -j)
 =  
 [ f_{p}(Y+ i_{0} + i_{*} - j_{0} -j) + f_{p}(Y) ] - f_{p}(Y)
\\
& \leq  
 \max\{f_{p}(Y+ i_{0} - j_{0}) + f_{p}(Y+ i_{*} -j),
 f_{p}(Y+ i_{0} -j ) + f_{p}(Y+ i_{*} - j_{0} )\} 
  - f_{p}(Y)  
\\
& \leq  
 f_{p}(Y+ i_{0} - j_{0}) +
\max\{f_{p}(Y+ i_{*}-j)  - f_{p}(Y),\ 
f_{p}(Y+ i_{*} - j_{0} ) - f_{p}(Y)\} 
\\
& < f_{p}(Y+ i_{0} - j_{0})  = f_{p}(Y'),
\end{align*}
where the first inequality is due to
(L3$[\mathbb{B}]$) 
of $\MncavlocS$,
and the second and third are 
by (\ref{vmexequiszl3-a1=01}) and (\ref{vmexequiszl3-c1-2=01}).
Hence follows (\ref{vmexequiszlem3-1=01}).

\medskip

Since $|X|=|Y'|$ and  $i_{*} \in X \setminus Y'$, 
 (\ref{vmexequiszlem3-1=01}) implies $(X, Y') \in \cal D$.
This contradicts the choice of $(X, Y)$, since  
$|X \setminus Y'| =|X \setminus Y| -1$.
Therefore, $\mathcal{D}$ must be empty, which means that
(P2$[\mathbb{B}]$) holds.
\end{proof}

Next we derive (P3$[\mathbb{B}]$) from (P1$[\mathbb{B}]$) and (P2$[\mathbb{B}]$).

\begin{lemmaM} \label{LMp12toP3}
{\rm (P1$[\mathbb{B}]$)} \&  {\rm (P2$[\mathbb{B}]$)} 
$\Longrightarrow$ 
{\rm (P3$[\mathbb{B}]$)}.
\end{lemmaM}

\begin{proof}
To prove (P3$[\mathbb{B}]$)
 by contradiction, we assume that
there exists a pair $(X,Y)$ for which 
(\ref{mnatP2=01}) 
fails.
That is, we assume that the set of such pairs
\[
\begin{array}{l}
 \mathcal{D} = \{(X, Y) \mid X, Y \in \dom  f,\ |X| < |Y|,\ 
\exists i_{*} \in X \setminus Y \  \mbox{s.t.} \\
\phantom{\mathcal{D} = \{(X, Y) \mid \ }
f(X) + f(Y)   > f(X- i_{*} +j ) + f(Y+ i_{*}-j )
\mbox{ for all } j \in Y \setminus X \}
\end{array}
\] 
is nonempty.
Take a pair $(X,Y) \in \mathcal{D}$ with 
$|X \bigtriangleup Y| = |X \setminus Y| + |Y \setminus X|$
minimum,
and fix $i_{*} \in X \setminus Y$ 
appearing in the definition of $\mathcal{D}$.
For a fixed $\varepsilon > 0$, define 
$p \in \mathbb{R}\sp{N}$ by  (\ref{mnatlocToP2pjdef}).
Then we have
(\ref{vmexequiszl3-c1-1=01}) 
and
(\ref{vmexequiszl3-c1-2=01})
in Claim 1 in the proof of Lemma \ref{LMmnatexcequicard01}.

\OMIT{
\[
\begin{array}{rcl}
p_{j} & = & \left\{
\begin{array}{ll}
f(X) - f(X- i_{*} +j) & (j \in Y \setminus X,\ 
         X- i_{*} +j \in \dom  f),\\
f(Y+ i_{*} -j) - f(Y) + \varepsilon \\
 & \hspace{-15mm} (j \in Y \setminus X,\ 
X- i_{*} +j \not\in \dom  f,\ 
Y+ i_{*} -j \in \dom  f),\\
0 & (\mbox{otherwise}).
\end{array}
\right.
\end{array}
\]
\medskip
Claim 1:
\begin{eqnarray}
f_{p}(X -  i_{*}  + j) & = & f_{p}(X) 
   \qquad (j \in Y \setminus X,\ X -  i_{*}  + j \in \dom  f),
         \label{vmexequiszl3-c1-1A}\\
f_{p}(Y +  i_{*}  - j) 
       & < & f_{p}(Y) \qquad (j \in Y \setminus X).
\label{vmexequiszl3-c1-2A}
\end{eqnarray}
\noindent (Proof of Claim~1) The equality
(\ref{vmexequiszl3-c1-1A}) is obvious from the definition of $p$.
 The inequality 
(\ref{vmexequiszl3-c1-2A}) follows from (\ref{vmexequiszl3-c1-1A}) and
$f_{p}(X) + f_{p}(Y) > f_{p}(X - i_{*} + j) + f_{p}(Y +i_{*} - j)$
if $X-i_{*}+j \in \dom f$.
If $X-i_{*}+j \not\in \dom f$,
(\ref{vmexequiszl3-c1-2A}) follows from the fact that
$f_{p}(Y +i_{*} - j) - f_{p}(Y) = -\varepsilon$ or $-\infty$
depending on whether $Y+i_{*}-j \in \dom f$ or not.
}

 \medskip

Claim 1:
There exists $j_{0} \in Y \setminus X$
such that $Y  - j_{0} \in \dom f$
and
\begin{equation}\label{vmexequiszl3-a1A}
f_{p}(Y - j_{0}) \geq  f_{p}(Y  - j)
\qquad
(j \in Y \setminus X).
\end{equation}

\noindent (Proof of Claim 1)
 Since $|X| < |Y|$, 
(P1$[\mathbb{B}]$)
implies that
there exists $j \in Y \setminus X$ 
such that $Y  - j \in \dom f$.
Any $j \in Y \setminus X$ that maximizes $f_{p}(Y  - j)$
serves as $j_{0}$.

\medskip

Claim 2:  
For $Y' = Y - j_0$ we have
\begin{equation}\label{vmexequiszlem3-1A}
 f_{p}(X) + f_{p}(Y') 
  > f_{p}(X -  i_{*}  + j) + f_{p}(Y' +  i_{*}  - j)
\qquad (j \in Y' \setminus X).
\end{equation}

\noindent (Proof of Claim 2)
Since this inequality is obvious when $X -i_{*} + j \not\in \dom f$,
we assume that $X - i_{*} + j \in \dom f$.
Then we have 
$f_{p}(X) = f_{p}(X - i_{*} + j)$ by 
(\ref{vmexequiszl3-c1-1=01}).
We also have 
\begin{align*}
& f_{p}(Y'+ i_{*} -j) =  
 [ f_{p}(Y + i_{*} - j_{0} -j) + f_{p}(Y) ] - f_{p}(Y)
\\
& \leq  
 \max\{f_{p}(Y - j_{0}) + f_{p}(Y+ i_{*} -j),
 f_{p}(Y -j ) + f_{p}(Y+ i_{*} - j_{0} )\} 
  - f_{p}(Y)  
\\
& \leq  
 f_{p}(Y - j_{0}) +
\max\{f_{p}(Y+ i_{*}-j)  - f_{p}(Y),\ 
f_{p}(Y+ i_{*} - j_{0} ) - f_{p}(Y)\} 
\\
& < f_{p}(Y - j_{0})  = f_{p}(Y'),
\end{align*}
where the first inequality is due to
(P1$[\mathbb{B}]$),
and the second and third inequalities are 
by (\ref{vmexequiszl3-a1A}) and 
(\ref{vmexequiszl3-c1-2=01}). 
Hence follows (\ref{vmexequiszlem3-1A}).

\medskip

By $|X| < |Y|$ and $|Y'| = |Y|-1$ we have $|X| \leq |Y'|$,
in which the possibility of equality is excluded.
Indeed, if $|X| = |Y'|$, then 
 (\ref{vmexequiszlem3-1A}) contradicts
(P2$[\mathbb{B}]$)
for $(X, Y', i_{*})$.
Therefore, $|X|<|Y'|$ holds.
 Hence, we have
$(X, Y') \in \cal D$ by 
 (\ref{vmexequiszlem3-1A}),
which is a contradiction to the choice of $(X, Y)$ since  
$|X' \bigtriangleup Y| \leq |X \bigtriangleup Y| -1$.
 Therefore, $\mathcal{D}$ must be empty, which means that
(P3$[\mathbb{B}]$) holds.
\end{proof}

Finally we derive (P4$[\mathbb{B}]$) from (P1$[\mathbb{B}]$) and (P2$[\mathbb{B}]$).

\begin{lemmaM} \label{LMp12toP4}
{\rm (P1$[\mathbb{B}]$)} \&  {\rm (P2$[\mathbb{B}]$)} 
$\Longrightarrow$ 
{\rm (P4$[\mathbb{B}]$)}.
\end{lemmaM}

\begin{proof}
To prove (P4$[\mathbb{B}]$)
 by contradiction, we assume that
there exists a pair $(X,Y)$ for which
(\ref{mnatconcavexc2})
 fails.
That is, we assume that the set of such pairs
\[
\begin{array}{l}
 \mathcal{D} = \{(X, Y) \mid X, Y \in \dom  f,\  
|X| > |Y|,\ \\
\phantom{\mathcal{D} = \{(X, Y) \mid \ }
\exists i_{*} \in X \setminus Y \ \mbox{s.t.} \  
f(X) + f(Y)  > f(X- i_{*}) + f(Y+ i_{*})
\\
\phantom{\mathcal{D} = \{(X, Y) \mid \ }
\mbox{and }
f(X) + f(Y)  > f(X- i_{*} +j ) + f(Y+ i_{*}-j )
\mbox{ for all } j \in Y \setminus X 
\}
\end{array}
\] 
is nonempty.
Take a pair $(X,Y) \in \mathcal{D}$ with 
$|X \bigtriangleup Y| = |X \setminus Y| + |Y \setminus X|$
minimum,
and fix $i_{*} \in X \setminus Y$ 
appearing in the definition of $\mathcal{D}$.

For a fixed $\varepsilon > 0$, define 
$p \in \mathbb{R}\sp{N}$ as follows.
 The component $p_{i_{*}}$ is defined by
\begin{align*} 
 &
\begin{array}{rcl}
p_{i_{*}}  & = & \left\{
\begin{array}{ll}
f(X) - f(X- i_{*})  & (X- i_{*} \in \dom  f),
\\
f(Y+ i_{*}) - f(Y) + \varepsilon 
 & (X- i_{*} \not\in \dom  f,\ Y+ i_{*} \in \dom  f),
\\
0 & (X- i_{*} \not\in \dom  f,\ Y+ i_{*} \not\in \dom  f).
\end{array}
\right.
\end{array}
\end{align*}
 The component $p_j$ for each $j \in Y \setminus X$ is defined by
\[
p_{j}  =  \left\{
\begin{array}{ll}
f(X) - f(X- i_{*} +j) + p_{i_{*}}  & 
(X- i_{*} +j \in \dom  f),
\\
f(Y+ i_{*} -j) - f(Y) + p_{i_{*}} + \varepsilon 
 &  (X- i_{*} +j \not\in \dom  f,\ Y+ i_{*} -j \in \dom  f),\\
0 & (X- i_{*} +j \not\in \dom  f,\ Y+ i_{*} -j \not\in \dom  f).
\end{array}
\right.
\]
 We set $p_j=0$ for all other components of $p$.

\medskip

Claim 1:
\begin{eqnarray}
f_{p}(X -  i_{*} ) & = & f_{p}(X)
   \qquad (X -  i_{*} \in \dom  f),
\label{vmexlocl3c13-b} \\
f_{p}(X -  i_{*}  + j) & = & f_{p}(X) 
   \qquad (j \in Y \setminus X,\ X -  i_{*}  + j \in \dom  f),
\label{vmexlocl3c11-b} \\
f_{p}(Y +  i_{*} )  & < & f_{p}(Y) , 
\label{vmexlocl3c14-b} \\
f_{p}(Y +  i_{*}  - j) 
       & < & f_{p}(Y) \qquad (j \in Y \setminus X).
\label{vmexlocl3c12-b}
\end{eqnarray}

\noindent (Proof of Claim~1) 
Similar to the proof of Claim 1 in the proof of Lemma \ref{LMmnatexcequicard01}.

\OMIT{
The equalities (\ref{vmexlocl3c13-b}) and (\ref{vmexlocl3c11-b}) 
are obvious from the definition of $p$.
If $X-i_{*} \in \dom f$,
(\ref{vmexlocl3c14-b}) 
follows from (\ref{vmexlocl3c13-b})
and $f_{p}(X) + f_{p}(Y) > f_{p}(X - i_{*}) + f_{p}(Y +i_{*})$;
otherwise,
(\ref{vmexlocl3c14-b}) follows from the fact that
$f_{p}(Y +i_{*}) - f_{p}(Y) = -\varepsilon$ or $-\infty$
depending on whether $Y+i_{*} \in \dom f$ or not.
Similarly,
(\ref{vmexlocl3c12-b}) follows from 
(\ref{vmexlocl3c11-b}) and
$f_{p}(X) + f_{p}(Y) > f_{p}(X - i_{*} + j) + f_{p}(Y +i_{*} - j)$
 if $X-i_{*}+j \in \dom f$;
otherwise,
(\ref{vmexlocl3c12-b}) follows from the fact that
$f_{p}(Y +i_{*} - j) - f_{p}(Y) = -\varepsilon$ or $-\infty$
depending on whether $Y+i_{*}-j \in \dom f$ or not.
}

\medskip

To write (\ref{vmexlocl3c13-b}) and (\ref{vmexlocl3c11-b})
in one formula, it is convenient to introduce 
a special symbol, say,
$\circ$  
to denote a null element such that 
$Z + \circ = Z$ 
for any $Z \subseteq N$.
Then (\ref{vmexlocl3c13-b}) and (\ref{vmexlocl3c11-b}) together
are expressed as 
\begin{eqnarray}
f_{p}(X -  i_{*}  + j) & = & f_{p}(X) 
   \qquad (j \in (Y \setminus X)\sp{\circ},
    \ X -  i_{*}  + j \in \dom  f),
\label{vmexlocl3c113-b} 
\end{eqnarray}
where, for any $Z \subseteq N$,
``$j \in Z\sp{\circ}$'' means that
$j \in Z$ or $j = \circ$.
Similarly, with the understanding that
 $Z - \circ = Z$ for any $Z \subseteq N$,
(\ref{vmexlocl3c14-b}) and (\ref{vmexlocl3c12-b}) are expressed  as 
\begin{eqnarray}
f_{p}(Y +  i_{*}  - j) 
       & < & f_{p}(Y) \qquad (j \in (Y \setminus X)\sp{\circ}).
\label{vmexlocl3c124-b}
\end{eqnarray}

\medskip

Claim 2:
There exists $i_{0} \in X \setminus Y - i_{*}$ and
$j_{0} \in (Y \setminus X)\sp{\circ}$ 
such that $Y + i_{0} - j_{0} \in \dom f$ and
\begin{equation}
\label{vmexlocl3a1-b}
f_{p}(Y + i_{0} - j_{0}) \geq  f_{p}(Y + i_{0} - j)
\qquad ( j \in (Y \setminus X)\sp{\circ} ).
\end{equation}

\noindent (Proof of Claim 2)
 Since $|X| > |Y|$, 
(P1$[\mathbb{B}]$)
implies the existence of $i_{0} \in X \setminus Y$ with
$ f(X) + f(Y) \le  f(X- i_{0} ) + f(Y +  i_{0})$,
where $i_{0} \ne i_{*}$ by $(X, Y) \in \mathcal{D}$.
This inequality implies $Y+i_{0}-j \in \dom f$ with $j = \circ$.
Any $j \in (Y \setminus X)\sp{\circ}$ 
that maximizes $f_{p}(Y + i_{0} - j)$
serves as $j_{0}$.

\medskip

Claim 3:
For $Y'= Y  + i_{0} - j_{0}$ we have
\begin{align} 
 f_{p}(X) + f_{p}(Y')  & > f_{p}(X -  i_{*}  + j) + f_{p}(Y' +  i_{*}  - j)
 \qquad (j \in (Y' \setminus X)\sp{\circ}).
\label{vmexloc32-b}
\end{align}

\noindent (Proof of Claim 3)
First note that 
$i_{*} \in X \setminus Y'$ by the choice of $Y'$.
 Since the inequality (\ref{vmexloc32-b}) is obvious 
when $X - i_{*} + j \not\in \dom f$,
we assume that $X -  i_{*}  + j \in \dom f$.
Then we have 
$f_{p}(X) = f_{p}(X - i_{*} + j)$ 
by (\ref{vmexlocl3c113-b}). 
We also have
\begin{align*}
& f_{p}(Y'+ i_{*} -j)
 =   [ f_{p}(Y+ i_{0} + i_{*} - j_{0} -j) + f_{p}(Y) ] - f_{p}(Y)
\\
& \leq  
 \max\{f_{p}(Y+ i_{0} - j_{0}) + f_{p}(Y+ i_{*} -j),
 f_{p}(Y+ i_{0} -j ) + f_{p}(Y+ i_{*} - j_{0} )\} 
  - f_{p}(Y)  
\\
& \leq  
 f_{p}(Y+ i_{0} - j_{0}) +
\max\{f_{p}(Y+ i_{*}-j)  - f_{p}(Y),\ 
f_{p}(Y+ i_{*} - j_{0} ) - f_{p}(Y)\} 
\\
& < f_{p}(Y+ i_{0} - j_{0})  = f_{p}(Y'),
\end{align*}
where the first inequality is by 
(P1$[\mathbb{B}]$) or (P2$[\mathbb{B}]$),
and the second and third inequalities follow
 from (\ref{vmexlocl3a1-b}) and (\ref{vmexlocl3c124-b}),
respectively.
Therefore, (\ref{vmexloc32-b}) holds.

\medskip

By $|X| > |Y|$ and $|Y'| \le |Y|+1$ we have $|X| \geq |Y'|$,
in which the possibility of equality is excluded.
Indeed, if $|X| = |Y'|$, then 
(\ref{vmexloc32-b}) contradicts (P2$[\mathbb{B}]$) for $(X, Y', i_{*})$.
Therefore, $|X|>|Y'|$ holds.
 Hence, we have
$(X, Y') \in \cal D$ by (\ref{vmexloc32-b}),
which is a contradiction to the choice of $(X, Y)$, since  
$|X' \bigtriangleup Y| \leq |X \bigtriangleup Y| -1$.
Therefore, $\mathcal{D}$ must be empty,
which means that 
(P4$[\mathbb{B}]$)
is satisfied.
\end{proof}


\subsection{Proof of Theorem~\ref{THmnatcavlocexc01hered}}
\label{SCproofmnatlocexc01hered}

We prove that, 
when $\dom f$ is an \Mnat-convex family containing the empty set,
$f$ is \Mnat-concave
if and only if it satisfies the first two conditions
(L1$[\mathbb{B}]$) and (L2$[\mathbb{B}]$) 
of the local exchange property $\MncavlocS$.
 By Theorem~\ref{THmnatcavlocexc01}, it suffices to show that
(L2$[\mathbb{B}]$)
implies (L3$[\mathbb{B}]$)
under this assumption.

To be specific, we show (\ref{mnatconcavexc22locAlpha}):
$ \alpha_{12}+\alpha_{34} \leq 
 \max\{\alpha_{13}+\alpha_{24},\alpha_{14}+\alpha_{23}\}$
in the proof of Proposition~\ref{PRmconcavlocexc01onlyif},
where $\alpha_{ij}=f(X+i+j)$.
We may assume 
$\alpha_{12} > - \infty$ and $\alpha_{34} > - \infty$,
since otherwise
the inequality holds trivially.
Then we have
$\alpha_{i} =f(X+i) > - \infty$ 
for $i \in \{ 1,2,3,4 \}$
by the assumption on $\dom f$,
whereas 
$\alpha_{ij} \in \Rminf$.
 By (L2$[\mathbb{B}]$) in (\ref{mnatconcavexc21loc}) we have
$ \alpha_{12}+\alpha_{3} \leq 
 \max\{\alpha_{13}+\alpha_{2},\alpha_{23}+\alpha_{1}\}$,
where
we may assume 
\begin{equation} \label{Mnatlocproofeqn3}
\alpha_{12}+\alpha_{3} \leq \alpha_{13}+\alpha_{2}
\end{equation}
by symmetry $1 \leftrightarrow 2$.
Consider the following three pairs of inequalities:
\begin{eqnarray}
 \alpha_{34}+\alpha_{2} & \leq & \alpha_{24}+\alpha_{3},
\label{Mnatlocproofeqn4}
\\
 \alpha_{34}+\alpha_{2} & \leq & \alpha_{23}+\alpha_{4};
\label{Mnatlocproofeqn42}
\\
 \alpha_{12}+\alpha_{4} & \leq & \alpha_{14}+\alpha_{2},
\label{Mnatlocproofeqn5}
\\
 \alpha_{12}+\alpha_{4} & \leq & \alpha_{24}+\alpha_{1};
\label{Mnatlocproofeqn52}
\\
 \alpha_{34}+\alpha_{1} & \leq & \alpha_{13}+\alpha_{4},
\label{Mnatlocproofeqn6}
\\
 \alpha_{34}+\alpha_{1} & \leq & \alpha_{14}+\alpha_{3} .
\label{Mnatlocproofeqn62}
\end{eqnarray}
The condition
(L2$[\mathbb{B}]$)  
implies the following:
(i) (\ref{Mnatlocproofeqn4}) or (\ref{Mnatlocproofeqn42}) (or both) holds,%
\footnote{
We cannot assume (\ref{Mnatlocproofeqn4})
based on the apparent symmetry $3 \leftrightarrow 4$
in (\ref{Mnatlocproofeqn4}) and (\ref{Mnatlocproofeqn42}),
because this symmetry is not present in (\ref{Mnatlocproofeqn3}).
} 
(ii)  (\ref{Mnatlocproofeqn5}) or (\ref{Mnatlocproofeqn52}) (or both) holds,
(iii)  (\ref{Mnatlocproofeqn6}) or (\ref{Mnatlocproofeqn62}) (or both) holds.
 Hence, it suffices to consider the following four cases:
Case 1: 
(\ref{Mnatlocproofeqn4}) holds;
Case 2: 
(\ref{Mnatlocproofeqn42}) and (\ref{Mnatlocproofeqn5}) hold; 
Case 3: 
(\ref{Mnatlocproofeqn42}), (\ref{Mnatlocproofeqn52}), and (\ref{Mnatlocproofeqn6}) hold;
Case 4: 
(\ref{Mnatlocproofeqn42}), (\ref{Mnatlocproofeqn52}), and (\ref{Mnatlocproofeqn62}) hold.

 In Case 1, the addition of 
(\ref{Mnatlocproofeqn3}) and (\ref{Mnatlocproofeqn4}) 
yields
$\alpha_{12}+\alpha_{34} \leq \alpha_{13}+\alpha_{24}$,
which implies (\ref{mnatconcavexc22locAlpha}).
In Case 2, the addition of 
(\ref{Mnatlocproofeqn42}) and (\ref{Mnatlocproofeqn5}) 
yields
$\alpha_{34} + \alpha_{12} \leq \alpha_{23}+\alpha_{14}$. 
In Case 3, the addition of 
 (\ref{Mnatlocproofeqn52}) and (\ref{Mnatlocproofeqn6}) 
yields
$\alpha_{12}+\alpha_{34} \leq \alpha_{24}+\alpha_{13}$.
In Case 4, the addition of 
(\ref{Mnatlocproofeqn3}), (\ref{Mnatlocproofeqn42}), (\ref{Mnatlocproofeqn52}), and (\ref{Mnatlocproofeqn62})
yields
\[
2 (\alpha_{12}+\alpha_{34})
  \leq  
\alpha_{13}+\alpha_{24}+\alpha_{14}+\alpha_{23}
 \leq  
2 \max\{\alpha_{13}+\alpha_{24},\alpha_{14}+\alpha_{23}\},
\]
which shows (\ref{mnatconcavexc22locAlpha}).
This completes the proof of Theorem~\ref{THmnatcavlocexc01hered}.

\subsection{Proof of Theorem~\ref{THmnatcavP1hered01}}
\label{SCproofmnatcavP1hered01}

We prove that,  when $\dom f$ contains the empty set,
$f$ is \Mnat-concave
if and only if it satisfies (P1$[\mathbb{B}]$).
The ``only-if'' part
is already shown in Theorem~\ref{THmconcavcardexc01} (1).
We prove the ``if'' part by means of 
a local exchange theorem, Theorem~\ref{THmnatcavlocexc01hered}
in Section~\ref{SCexchange01loc}, 
which has already been proved in Section~\ref{SCproofmnatlocexc01hered}.

The first two conditions
(L1$[\mathbb{B}]$) and (L2$[\mathbb{B}]$) 
of the local exchange property $\MncavlocS$
are immediate consequences of (P1$[\mathbb{B}]$).
In addition, $\dom f$ is an \Mnat-convex family, as shown below.
Therefore, $f$ is \Mnat-concave by Theorem~\ref{THmnatcavlocexc01hered}.

For \Mnat-convexity of $\dom f$,
we show that $\mathcal{F}= \dom f$
satisfies the axioms for independent sets of a matroid:
\hbox{(I-1)} 
$ \emptyset \in \mathcal{F}$,
\ 
\hbox{(I-2)}  
$X \subseteq  Y \in \mathcal{F} \  \Rightarrow \  X \in \mathcal{F}$,
\ 
\hbox{(I-3)} 
 $X, Y\in \mathcal{F}, \ 
     |X| <|Y|  \  \Rightarrow \  X + j \in \mathcal{F}$  
         for some $j \in Y \setminus X$.
Here (I-1) holds by assumption and 
(I-3) is immediate from (P1$[\mathbb{B}]$).
The second property (I-2) can be shown 
as follows.
By condition (P1$[\mathbb{B}]$) we have:
\begin{align} 
& 
X, Y \in \mathcal{F}, \ X \subsetneqq Y
 \ \Longrightarrow \
\mbox{there exists $j \in Y \setminus X$
such that $X + j, Y  - j  \in \mathcal{F}$,
}
\label{mnsetP1}
\end{align}
which implies,
by Lemma~\ref{LMgenbox01} below,
that
\begin{align} 
&
X, Y \in \mathcal{F},  \  X \subseteq Z \subseteq  Y
 \ \Longrightarrow \
Z \in \mathcal{F}.
\label{mnsetbox}
\end{align}
Since $\emptyset \subseteq X \subseteq  Y$
and $\emptyset, Y \in \mathcal{F}$ in (I-2),
we have $X \in \mathcal{F}$.
This completes the proof of Theorem~\ref{THmnatcavP1hered01}.

\begin{lemmaM}  \label{LMgenbox01}
If a set family $\mathcal{F}$ satisfies 
{\rm (\ref{mnsetP1})}, then it satisfies {\rm (\ref{mnsetbox})}.
\end{lemmaM}
\begin{proof}
We show (\ref{mnsetbox}) by induction on $|Y \setminus X|$.
(\ref{mnsetbox}) is trivially true if $|Y \setminus X| \leq 1$.
Assume $|Y \setminus X| \geq 2$,
 take $j \in Y \setminus X$ in (\ref{mnsetP1}),  and
consider $Z$ satisfying $X \subsetneqq Z \subsetneqq Y$.
If $j \in Z$, then $X + j \subseteq Z \subseteq Y$
with $|Y \setminus (X+j)| = |Y \setminus X| -1$.
If $j \not\in Z$, then $X \subseteq Z \subseteq Y-j$
with $|(Y-j) \setminus X| = |Y \setminus X| -1$.
In either case we obtain $Z \in \mathcal{F}$ by the induction hypothesis.
\end{proof}


\section{M-concave Set Functions (Valuated Matroids)}
\label{SCmcavsetfn}

\subsection{Definition}

Let $f: 2\sp{N} \to \Rminf$ be
a real-valued set function on 
$N = \{ 1,2,\ldots, n \}$ and 
$\mathcal{B} = \dom f$ be the effective domain of $f$.

We say that a function
$f$ 
is a {\em valuated matroid} (or {\em matroid valuation}), 
if, for any $X, Y \in \mathcal{B}$ and $i \in X \setminus Y$,
there exists some $j \in Y \setminus X$ such that
$X - i +j \in \mathcal{B}$, $ Y + i -j  \in \mathcal{B}$ and
\begin{equation}  \label{valmatexc1}
f( X) + f( Y ) \leq  f( X - i + j) + f( Y + i -j).
\end{equation}
This property is referred to as the {\em exchange property}.
A valuated matroid is also called an {\em M-concave set function}.
In this paper we use this terminology 
to emphasize its concavity aspects.

An M-concave function
can also be defined 
without explicit reference to its effective domain
by considering 
by the following expression of the exchange property:%
\begin{description}
\item[\McavSb] 
For any $X, Y \subseteq N$ and $i \in X \setminus Y$, we have
\begin{align}
f( X) + f( Y )   &\leq 
 \max_{j \in Y \setminus X}  \{ f( X - i + j) + f( Y + i -j) \} .
\label{valmatexc2}
\end{align}
\end{description}

The effective domain of an M-concave function
is equipped with a nice combinatorial structure.
Let $\mathcal{B}$ denote the effective domain of an M-concave function $f$.
As a consequence of the exchange property $\McavS$ of function $f$,
the set family $\mathcal{B}$
satisfies the following exchange property:
\begin{description}
\item[\BvexSb]
For any $X, Y \in \mathcal{B}$ and $i \in X \setminus Y$, 
there exists some $j \in Y \setminus X$ such that
$X - i +j \in \mathcal{B}$ and $ Y + i -j  \in \mathcal{B}$.
\end{description}
This means that $\mathcal{B}$ forms the family of bases of a matroid.
We often refer to 
a set family $\mathcal{B}$ as an {\em M-convex family}
if it is a nonempty family and satisfies the exchange property $\BvexS$.
Therefore, an M-convex family is a synonym of the base family of a matroid.

A set family $\mathcal{B}$ satisfying $\BvexS$
consists of equi-cardinal subsets, that is,
\begin{align} 
X, Y \in \mathcal{B} 
\ \Longrightarrow \
|X| = |Y|
\label{msetequicard}
\end{align}
holds if $\mathcal{B}$ satisfies the exchange property $\BvexS$.
In view of the importance in our context,
we state this fact as a proposition with a formal proof,
although this is well known in matroid theory.

\begin{propositionM} \label{PRmsetequicard01}
An M-convex family consists of equi-cardinal subsets.
\end{propositionM}
\begin{proof}
We show (\ref{msetequicard})
by induction on 
$|X \bigtriangleup Y| = |X \setminus Y| + |Y \setminus X|$.
If $|X \bigtriangleup Y| = 0$, (\ref{msetequicard}) is trivially true.
If $|X \bigtriangleup Y| \geq 1$,
we may assume $X \setminus Y \not= \emptyset$.
Take any $i \in X \setminus Y$. 
By $\BvexS$, $Y + i - j \in \mathcal{B}$ for some  $j \in Y \setminus X$.
For $Y' = Y + i - j $ we have
$|X \bigtriangleup Y'| < |X \bigtriangleup Y|$,
and hence $|X| = |Y'|$ by the induction hypothesis.
Since $|Y| = |Y'|$, this shows $|X| = |Y|$.
\end{proof}

\begin{remark} \rm  \label{RMvalmat}
The concept of valuated matroid is introduced by Dress--Wenzel \cite{DW90=valmat,DW92=valmat}.
The subsequent development leading to discrete convex analysis
is expounded in Murota \cite[Chapter 5]{Mspr2000=valmat}.
\finbox
\end{remark}

\subsection{Relation between M- and \Mnat-concave functions}
\label{SCrelmmncavsetfn}

We show that, while M-concave functions 
are a special case of 
\Mnat-concave functions,
they are in fact equivalent concepts
in the sense to be formulated in 
Proposition~\ref{PRmnatequicardvalmat}.

First, M-concave functions form a subclass of
\Mnat-concave functions with equi-cardinal effective domains.

\begin{propositionM} \label{PRmcav=mnatcav+equicard}
A set function $f$ is M-concave
 if and only if it is an \Mnat-concave function and 
$|X| = |Y|$ for all $X, Y \in \dom f$.
\end{propositionM}
\begin{proof}
As is already noted, the effective domain 
of an M-concave function consists of equi-cardinal sets
(Proposition~\ref{PRmsetequicard01}).
For a function $f$ with equi-cardinal $\dom f$,
$\MncavS$ is equivalent to $\McavS$
since (\ref{mnatcav1}) cannot happen.
\end{proof}

Second, we discuss the essential equivalence of M- and \Mnat-concave functions.
For a function
$f: 2^{N} \to \Rminf$,
we associate a function 
$\tilde{f}$
on an equi-cardinal family on a larger set $\tilde{N}$. 
Denote by $r$ and $r'$ the maximum and minimum, respectively, of $|X|$ for $X \in \dom f$.
Let
$s \geq r-r'$ and
$S = \{ n+1,n+2,\ldots, n+s \}$.
We enlarge the underlying set to  
$\tilde{N} = N \cup S = \{ 1,2,\ldots, \tilde n \}$,
where $\tilde n =n+s$, and define
$\tilde{f}: 2^{\tilde N} \to \Rminf$ by
\begin{align} \label{assocMdef} 
\tilde{f}(Z)  =
   \left\{  \begin{array}{ll}
    f(Z \cap N)         &   (|Z| = r) ,     \\
   -\infty    &   (\mbox{otherwise}) ,  \\
                     \end{array}  \right.
\end{align}
for which $\dom \tilde{f}$ is an equi-cardinal family.
For $X \subseteq N$ and $U \subseteq S$,
we have $\tilde{f}(X \cup U) = f(X)$ 
if $|X|+|U|=r $.
Therefore, if we want to maximize $f$, for example, we may maximize 
the associated function $\tilde{f}$  to obtain an optimal solution for $f$.

We illustrate the construction (\ref{assocMdef}) by a simple example.

\begin{example} \rm  \label{EXmnatcvU32toM}
Consider a function $f$ 
defined on $N=\{ 1,2,3 \}$ 
as
$f(\emptyset)=f(\{ 1 \})=0$,
$f(\{ 2 \})=f(\{ 3 \})=f(\{ 1,2 \})=f(\{ 1,3 \})=f(\{ 2,3 \})=1$,
and
$f(\{ 1,2,3 \})=-\infty$.
We have 
$r=2$ and $r'=0$, and hence
we can take $s =2$, $S = \{ 4,5 \}$, and $\tilde{N} = \{ 1,2,3,4,5 \}$.
The corresponding function $\tilde{f}$ is given by 
\begin{align*}
& \tilde{f}(\{ 4,5 \})=\tilde{f}(\{ 1,k \})=0,
\quad
\tilde{f}(\{ 2,k \}) =\tilde{f}(\{ 3,k \}) =1
\quad (k=4,5),
\\
& 
\tilde{f}(\{ 1,2 \}) =\tilde{f}(\{ 1,3 \}) =\tilde{f}(\{ 2,3 \}) =1,
\quad
\tilde{f}(\{ 1,2,3 \})=-\infty.
\end{align*}
This function $f$ is \Mnat-concave, satisfying the condition $\MncavS$,
while the corresponding function $\tilde{f}$ is M-concave, 
satisfying the condition $\McavS$.
\finbox
\end{example}

\begin{propositionM}[\cite{Mmultexcstr18=valmat}]  \label{PRmnatequicardvalmat}
A set function $f$ is \Mnat-concave 
if and only if $\tilde{f}$ is M-concave.
\end{propositionM}
\begin{proof}
The exchange property \McavS for $\tilde{f}$ amounts to the following,
where $X, Y \in \dom f$ and  $U, V \subseteq S$ with
$|X|+|U|=|Y|+|V|=r$.
\begin{itemize}
\item
For any $i \in X \setminus Y$
there exists  
$j \in Y \setminus X$ with
\begin{align}
&  \tilde f( X \cup U) +\tilde  f( Y \cup V )  
\leq   \tilde f( (X - i + j )  \cup U  ) + \tilde f( (Y + i -j) \cup V  ),
\label{assocM11}
\end{align}
or there exists  
$j \in V \setminus U$ with 
\begin{align}
&  \tilde f( X \cup U) +\tilde  f( Y \cup V )  
\leq   \tilde f( (X - i)  \cup (U +j)  ) + \tilde f( (Y + i) \cup (V-j)  ).
\label{assocM12}
\end{align}

\item
For any $i \in U \setminus V$
there exists  
$j \in Y \setminus X$ with
\begin{align}
&  \tilde f( X \cup U) +\tilde  f( Y \cup V )  
\leq   \tilde f( (X + j )  \cup (U -i)  ) + \tilde f( (Y -j) \cup (V+i)  ),
\label{assocM21}
\end{align}
or there exists  
$j \in V \setminus U$ with
\begin{align}
&  \tilde f( X \cup U) +\tilde  f( Y \cup V )  
\leq   \tilde f( X  \cup (U -i +j)  ) + \tilde f( Y  \cup (V+i-j)  ).
\label{assocM22}
\end{align}
\end{itemize}
Suppose that $f$ is \Mnat-concave.
For any $i \in X \setminus Y$
we have (\ref{mnatcav1}) or (\ref{mnatcav2}).
In the case of (\ref{mnatcav2}) we obtain (\ref{assocM11}).
In the case of (\ref{mnatcav1}) we obtain (\ref{assocM12})
for any $j \in V \setminus U$, if $V \setminus U$ is nonempty.
If $V \setminus U$ is empty, then $|X| \leq |Y|$ and
we have (\ref{assocM11}) by 
(P2$[\mathbb{B}]$) and (P3$[\mathbb{B}]$).
Next, take any $i \in U \setminus V$ (when $U \setminus V \not= \emptyset$).
If $V \setminus U$ is nonempty, (\ref{assocM22}) holds 
for any $j \in V \setminus U$.
If $V \setminus U$ is empty, we have
$|U| > |V|$ and hence $|X| < |Y|$.
Then (P1$[\mathbb{B}]$) shows (\ref{assocM21}).
Thus we have shown the ``only-if'' part.
The converse (``if'' part) is also true, since
(\ref{mnatcav1}) follows from (\ref{assocM12}),
and (\ref{mnatcav2}) from (\ref{assocM11}).
\end{proof}

\subsection{Exchange properties characterizing M-concave functions}
\label{SCmexch01}


The local exchange property for M-concave set functions 
reads as follows.

\begin{description}
\item[\McavlocSb] 
For any $X, Y \subseteq N$ with $|X \setminus Y | = 2$, there 
exist $i \in X \setminus Y$ and $j \in Y \setminus X$ such that
\begin{equation}  \label{valmatexc1loc}
f( X) + f( Y ) \leq  f( X - i + j) + f( Y + i -j).
\end{equation}
\end{description}

\begin{theorem}\label{THmcavlocexc01}
A set function  $f: 2\sp{N} \to \Rminf$ is M-concave
if and only if 
$\dom f$ is a matroid basis family (an M-convex family) and 
$\McavlocS$
is satisfied.
\end{theorem}
\begin{proof}
By Proposition~\ref{PRmcav=mnatcav+equicard},
an M-concave function is precisely
an \Mnat-concave function with an equi-cardinal effective domain.
We use Theorem~\ref{THmnatcavlocexc01} that characterizes \Mnat-concavity 
by the local exchange property $\MncavlocS$.
If $\dom f$ is equi-cardinal,
the first two conditions 
(L1$[\mathbb{B}]$) and (L2$[\mathbb{B}]$) in $\MncavlocS$
are satisfied trivially,  since the left-hand sides 
of (\ref{mnatconcavexc20loc}) and (\ref{mnatconcavexc21loc})
are always equal to $-\infty$.
The third condition (L3$[\mathbb{B}]$) in $\MncavlocS$
is equivalent to $\McavlocS$.
\end{proof}

\begin{remark} \rm  \label{RMmlocdomcond}
In Theorem~\ref{THmcavlocexc01},
the assumption on $\dom f$ is indispensable.
For example, 
let $N=\{ 1,2,\ldots, 6 \}$
and define 
$f(\{ 1,2,3 \}) = f(\{ 4,5,6 \})=0$,
and $f(X)=-\infty$ for $X \not= \{ 1,2,3 \}, \{ 4,5,6 \}$.
This function $f$ is not M-concave, since 
$\dom f = \{ \{ 1,2,3 \}, \{ 4,5,6 \}  \}$ is not an M-convex family. 
However, $f$ satisfies the condition 
\McavlocS
in a trivial manner, since the left-hand side of 
\eqref{valmatexc1loc}
is equal to $-\infty$ whenever $|X \setminus Y | = 2$.
\finbox
\end{remark}

We consider another (seemingly) weaker exchange property:
\begin{description}
\item[\McavwSb] 
For any distinct $X, Y \subseteq N$, there exist $i \in X \setminus Y$
and $j \in Y \setminus X$ that satisfy \eqref{valmatexc1loc}.
\end{description}
If $f$ has this property, its effective domain $\mathcal{B} = \dom f$ satisfies
\begin{description}
\item[\BvexwSb] 
For any distinct $X, Y \in \mathcal{B}$, there exist $i \in X \setminus Y$
and $j \in Y \setminus X$ such that
$X - i +j \in \mathcal{B}$ and $ Y + i -j  \in \mathcal{B}$.
\end{description}

The seemingly weaker condition $\McavwS$ is, in fact, equivalent to $\McavS$.

\begin{theorem}\label{THmcavexcweak01}
A set function  $f: 2\sp{N} \to \Rminf$ is M-concave
if and only if \McavwS is satisfied.
\end{theorem}
\begin{proof}
The implication 
``$\McavS \Rightarrow \McavwS$''
is obvious.
To prove the converse, assume $\McavwS$ for $f$.
Then $\dom f$ has the exchange property $\BvexwS$.
It is known \cite[Theorem 2.3.14]{Mspr2000=valmat}
that $\BvexwS$ is equivalent to $\BvexS$.%
\footnote{
$\BvexwS$ and $\BvexS$ here correspond, respectively,  to 
(BM$_{\pm\rm w}$) and (BM$_{\pm}$)
in \cite{Mspr2000=valmat}.
}  
Then the claim follows from Theorem~\ref{THmcavlocexc01}.
\end{proof}

\begin{remark} \rm  \label{RMmlocweakcond}
Theorem~\ref{THmcavlocexc01} is due to 
Dress--Wenzel \cite{DWperf92=valmat} and Murota \cite{Mmax97=valmat},
and 
Theorem~\ref{THmcavexcweak01} is to Murota \cite{Mmax97=valmat}.
See also \cite[Theorem 5.2.25]{Mspr2000=valmat},
where the exchange properties \McavlocS and \McavwS
are called (VM$_{\rm loc}$) and (VM$_{\rm w}$), respectively. 
\finbox
\end{remark}


\section{Multiple Exchange Properties}
\label{SCexchange01mult}


\subsection{Theorems of multiple exchange properties}

As a generalization of the exchange property \MncavS 
for \Mnat-concave functions
we may conceive two versions of {\em multiple exchange property}:
\begin{description}
\item[\MncavmSb]
For any $X, Y \subseteq N$ and $I \subseteq X \setminus Y$,
there exists $J \subseteq Y \setminus X$ such that
\begin{align}
f( X) + f( Y )   \leq 
  f((X \setminus I) \cup J) +f((Y \setminus J) \cup I)   ,
\label{mnatconcavexcmult}
\end{align}
\item[\MncavmsSb]
For any $X, Y \subseteq N$ and $I \subseteq X \setminus Y$,
there exists $J \subseteq Y \setminus X$ with $|J| \leq |I|$ and (\ref{mnatconcavexcmult}),
\end{description}
where the latter,
requiring the cardinality condition $|J| \leq |I|$ on $J$,
is a stronger property than the former.
Thus the stronger form $\MncavmsS$ implies the weaker form $\MncavmS$.
The subscript ``m'' stands for ``multiple'' and ``s'' for ``stronger.''

Obviously, the (ordinary) exchange property \MncavS
follows from 
the stronger form \MncavmsS 
as its special case with $|I|=1$,
but it does not immediately follow from the weaker form $\MncavmS$.
These three conditions are, in fact, equivalent,
as is stated in the following theorem.

\begin{theorem} \label{THmultexchmnat}
For a function
$f: 2\sp{N} \to \Rminf$ 
with $\dom f \not= \emptyset$,
the three conditions  $\MncavS$, $\MncavmS$, and $\MncavmsS$
are pairwise equivalent.
Therefore, every M$\sp{\natural}$-concave set function
has the stronger multiple exchange property $\MncavmsS$.
\end{theorem}
\begin{proof}
The proof is given in Section~\ref{SCproofmnatmult01}.
It is based on the Fenchel-type duality theorem
\cite[Theorem~8.21]{Mdcasiam=valmat}.
\end{proof}

For M-concave functions, the multiple exchange property 
takes the following form,
since the effective domain is equi-cardinal.

\begin{description}
\item[\McavmSb]
For any $X, Y \subseteq N$ and $I \subseteq X \setminus Y$,
there exists $J \subseteq Y \setminus X$ with $|J| = |I|$ and (\ref{mnatconcavexcmult}).
\end{description}

\begin{theorem} \label{THmultexchvalmat}
Every M-concave function (valuated matroid) 
has the multiple exchange property $\McavmS$.
\end{theorem}
\begin{proof}
This follows from the implication
``$\MncavS \Rightarrow \MncavmS$''
in Theorem~\ref{THmultexchmnat}. 
using Proposition~\ref{PRmcav=mnatcav+equicard}.
See the proof of Lemma \ref{LMmexc0m} in Section~\ref{SCproofmnatmult01} for detail.
\end{proof}

\begin{remark} \rm  \label{RMmultexchbib}
Theorem~\ref{THmultexchvalmat} 
and the equivalence of 
$\MncavS$ and $\MncavmS$
in Theorem~\ref{THmultexchmnat}
are due to Murota \cite{Mmultexc18=valmat},
whereas the equivalence of 
$\MncavS$ and to the stronger version $\MncavmsS$
is established in Murota \cite{Mmultexcstr18=valmat}.
\finbox
\end{remark}

\begin{remark} \rm  \label{RMmultBexc}
Theorem~\ref{THmultexchvalmat} implies a classical result in matroid theory
(cf., Kung \cite{Kun86b=valmat}, Schrijver \cite[Section 39.9a]{Sch03=valmat})
that the base family $\mathcal{B}$ of a matroid has the multiple exchange property:
\begin{description}
\item[\BvexmSb] 
For any $X, Y \in \mathcal{B}$ and $I \subseteq X \setminus Y$,
there exists $J \subseteq Y \setminus X$ 
with $|J| = |I|$ such that
$(X \setminus I) \cup J \in \mathcal{B}$
and $(Y \setminus J) \cup I \in \mathcal{B}$.
\end{description}
It follows from Theorem~\ref{THmultexchmnat} 
that a nonempty family $\mathcal{F} \subseteq 2\sp{N}$
satisfies \BnvexS
if and only if it satisfies
the multiple exchange property:
\begin{description}
\item[\BnvexmSb] 
For any $X, Y \in \mathcal{F}$ and $I \subseteq X \setminus Y$,
there exists $J \subseteq Y \setminus X$ such that
$(X \setminus I) \cup J \in \mathcal{F}$
and $(Y \setminus J) \cup I \in \mathcal{F}$
\end{description}
as well as its stronger form with an additional condition $|J| \leq |I|$ on $J$.
Therefore, every g-matroid has this multiple exchange property.
\finbox
\end{remark}

\begin{remark} \rm  \label{RMmulteco}
The multiple exchange property $\MncavmS$ here
is the same as the ``strong no complementarities property (SNC)''
introduced by Gul--Stacchetti \cite{GS99=valmat},
where it is shown that (SNC) implies the gross substitutes property (GS)
of Kelso--Crawford \cite{KC82=valmat}.
By a result of Fujishige--Yang \cite{FY03gs=valmat},
on the other hand,
(GS) is equivalent to $\MncavS$.
Therefore, Theorem~\ref{THmultexchmnat} above 
reveals that (SNC) is equivalent to (GS).
\finbox
\end{remark}

\subsection{Proof of Theorem~\ref{THmultexchmnat}}
\label{SCproofmnatmult01}

We prove Theorem~\ref{THmultexchmnat}
about multiple exchange properties.
Our proof first shows 
the equivalence of $\MncavS$ and $\MncavmS$
in Lemmas \ref{LMmnatexc0m} and \ref{LMmnatexcm0}.
Using this we further show the implication
``$\MncavS \Rightarrow \MncavmsS$''
in Lemma \ref{LMmnatexc0ms}.
The converse 
``$\MncavS \Leftarrow \MncavmsS$''
is obvious, 
as is already mentioned in before Theorem \ref{THmultexchmnat}.

\begin{lemmaM}  \label{LMmnatexc0m}
$\MncavmS$ implies $\MncavS$.
\end{lemmaM}
\begin{proof}
First, it can be shown that $\dom f$ satisfies $\BnvexS$;
see \cite[Section 5.1]{Mmultexc18=valmat} for the detail.
Then the proof is reduced,
by Theorem~\ref{THmnatcavlocexc01},
 to showing the local exchange property $\MncavlocS$
in Section~\ref{SCexchange01loc}.
The first two conditions
(L1$[\mathbb{B}]$) and (L2$[\mathbb{B}]$)
of $\MncavlocS$
are immediate from $\MncavmS$,
and third condition (L3$[\mathbb{B}]$)
can be shown similarly 
to the proof of Theorem~\ref{THmnatcavlocexc01hered};
see \cite[Section 5.2]{Mmultexc18=valmat} for the detail.
\end{proof}

\begin{lemmaM}  \label{LMmnatexcm0}
$\MncavS$ implies $\MncavmS$.
\end{lemmaM}
\begin{proof}
Let $f: 2^{N} \to \Rminf$ be 
an M$^{\natural}$-concave function,
which, by definition,  satisfies the exchange property $\MncavS$.
Let  $X, Y \in \dom f$ and $I \subseteq X \setminus Y$.
With the notations:
\begin{align}
 &C = X \cap Y,
\qquad
 X_{0} = X \setminus Y = X \setminus C,
\qquad
 Y_{0} = Y \setminus X = Y \setminus C ,
\label{mexcCX0Y0def}
\\
& f_{1}(J) =  f((X \setminus I) \cup J) 
  = f( (X_{0} \setminus I) \cup C \cup J)
\qquad (J \subseteq Y_{0}),
\label{mexcf1def}
\\
& f_{2}(J) = f((Y \setminus J) \cup I)
  = f(  I \cup C \cup (Y_{0} \setminus J) )
\qquad (J \subseteq Y_{0}),
\label{mexcf2def}
\end{align}
the multiple exchange property \MncavmS
is rewritten as 
\begin{align}
f( X) + f( Y )   \leq 
 \max_{J \subseteq Y_{0}}  \{ f_{1}(J) + f_{2}(J) \}.
\label{mnatconcavexcmult3}
\end{align}
Both  $f_{1}$ and $f_{2}$
are M$^{\natural}$-concave set functions on $Y_{0}$,
where the nonemptiness of $\dom f_{1}$ and $\dom f_{2}$
can be shown by induction on $|I|$ using $\BnvexS$.

For $i=1,2$, let $g_{i}$ be the (convex) conjugate function of $f_{i}$ defined as
\begin{align*}
g_{1}(q) &=  
 \max_{J \subseteq Y_{0}} \{  f_{1}(J) - q(J) \} 
\qquad (q \in \RR^{Y_{0}}),
\\
g_{2}(q) &= 
 \max_{J \subseteq Y_{0}} \{  f_{2}(J) - q(J) \} 
\qquad (q \in \RR^{Y_{0}}),
\end{align*}
where $q(J) = \sum_{j \in J} q_{j}$.
The Fenchel-type duality
\cite[Theorem~8.21(1)]{Mdcasiam=valmat}
shows%
\footnote{
The assumption $\dom g_{1} \cap \dom g_{2} \not= \emptyset$
in Murota \cite[Theorem 8.21 (1)]{Mdcasiam=valmat}
is satisfied,
since $\dom g_{1} = \dom g_{2} = \RR^{N}$.
} 
\begin{equation} \label{mnatconcavexcmult3fenc}
\max_{J \subseteq Y_{0}}  \{ f_{1}(J) + f_{2}(J) \}
= \inf_{q \in \RR^{Y_{0}}} \{ g_{1}(q) + g_{2}(-q) \},
\end{equation}
where the maximum on the left-hand side is
defined to be $-\infty$ 
if $\dom f_{1} \cap \dom f_{2} = \emptyset$.
Combining \eqref{mnatconcavexcmult3fenc} with Lemma \ref{LMg1qg2q} below, we obtain 
\[
\max_{J \subseteq Y_{0}}  \{ f_{1}(J) + f_{2}(J) \}
= \inf_{q \in \RR^{Y_{0}}} \{ g_{1}(q) + g_{2}(-q) \}
\geq f( X) + f( Y ),
\]
which shows the desired inequality (\ref{mnatconcavexcmult3})
as well as the finiteness of the value of \eqref{mnatconcavexcmult3fenc}.
\end{proof}

\begin{lemmaM} \label{LMg1qg2q}
For any $q \in \RR^{Y_{0}}$, we have
$g_{1}(q) + g_{2}(-q) \geq f( X) + f( Y )$.
\end{lemmaM}
\begin{proof}
Let $g$ be the (convex) conjugate function of $f$, i.e.,
\begin{equation} \label{gpdef}
g(p) =  \max_{Z \subseteq N} \{  f(Z) - p(Z) \} 
\qquad (p \in \RR^{N}).
\end{equation}

\begin{figure}\begin{center}
\includegraphics[height=30mm]{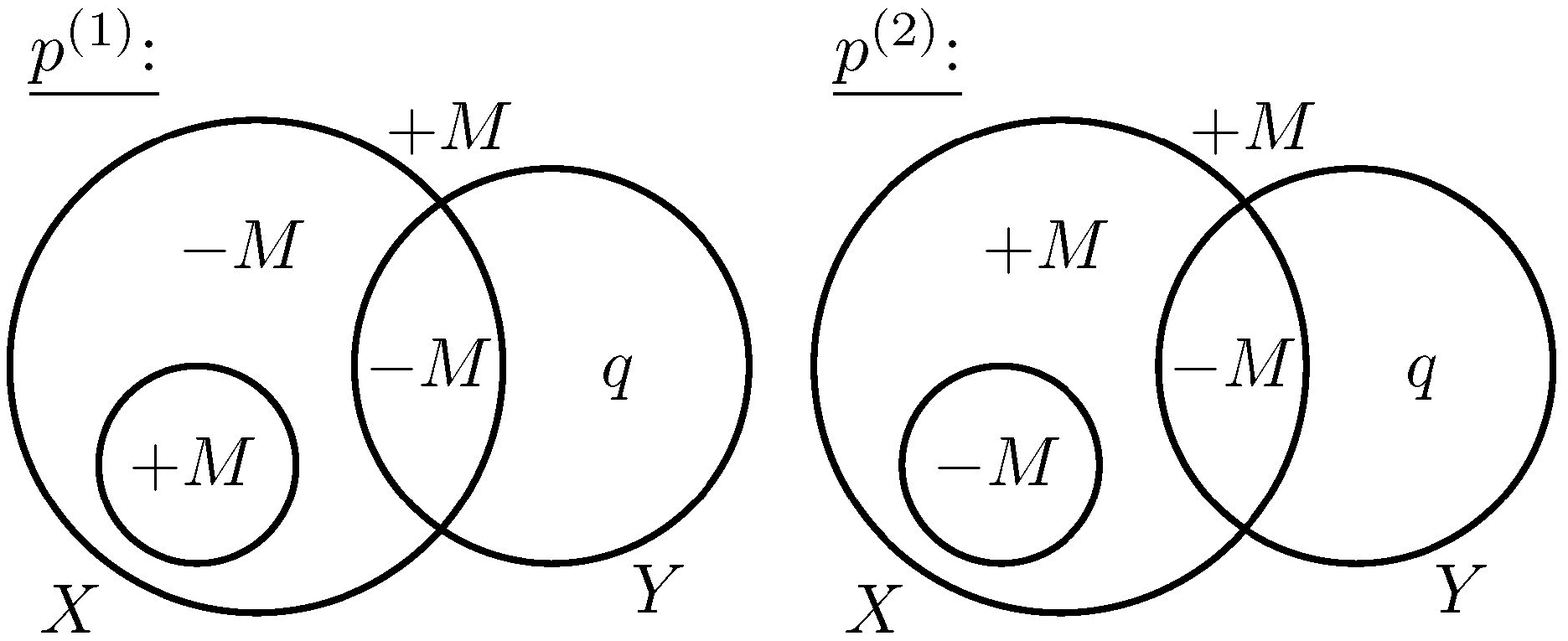}
\quad
\includegraphics[height=30mm]{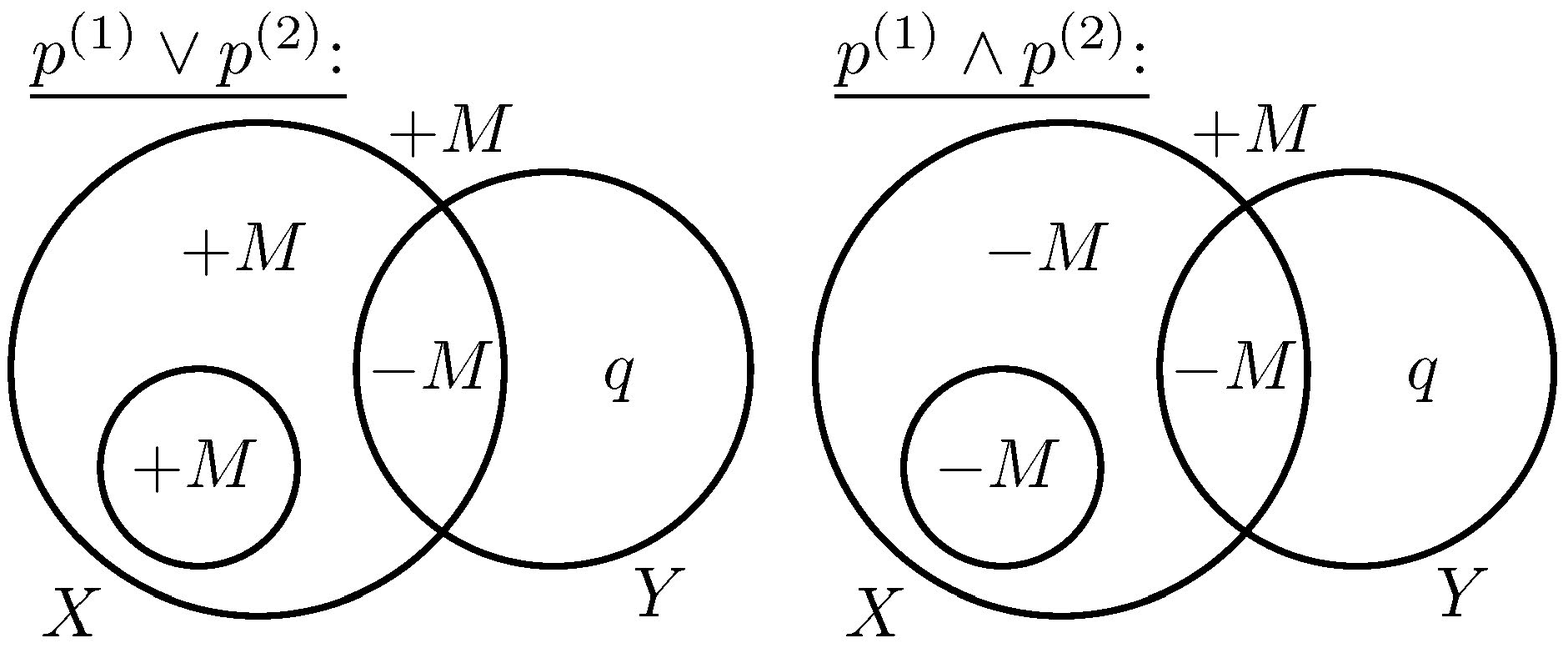}
\vspace{0.5\baselineskip}
\caption{Vectors $p^{(1)}$, $p^{(2)}$, $p^{(1)} \vee p^{(2)}$, and $p^{(1)} \wedge p^{(2)}$}
	\label{FGp1p2def}
\end{center}\end{figure}

For a vector $q \in \RR^{Y_{0}}$ we define 
$p^{(1)}, p^{(2)} \in \RR^{N}$ by
\begin{align*} 
p^{(1)}_{i}  &= p^{(2)}_{i}  =
   \left\{  \begin{array}{ll}
    q_{i}          &   (i  \in Y_{0}) ,     \\
   - M     &   (i \in C ),  \\
   + M     &   (i \in N \setminus (X \cup Y)  ) , \\
                     \end{array}  \right.
\quad 
p^{(1)}_{i}  = - p^{(2)}_{i}  =
   \left\{  \begin{array}{ll}
   - M     &   (i \in X_{0} \setminus I ),  \\
   + M     &   (i \in I ),  \\
                     \end{array}  \right.
\end{align*}
where $M$ is a sufficiently large positive number
(see Fig.~\ref{FGp1p2def}).
Then the maximizer $Z$ of $g(p)$ in (\ref{gpdef}) 
for $p = p^{(1)}$ 
must include $(X_{0} \setminus I) \cup C$
and avoid $I \cup ( N \setminus (X \cup Y) )$.
For $p = p^{(2)}$, the maximizer $Z$ 
must include $I \cup C$ and avoid 
$( X \setminus (I \cup C) ) \cup ( N \setminus (X \cup Y) )$.
Therefore,  we have
\begin{align*}
g_{1}(q) &=  
 \max_{J \subseteq Y_{0}} \{ f( (X_{0} \setminus I) \cup C \cup J) - q(J) \} 
\\ & =
g(p^{(1)}) - M (|X_{0} \setminus I|+|C|),
\\
g_{2}(-q) &= 
 \max_{J \subseteq Y_{0}} \{  
 f(  I \cup C \cup (Y_{0} \setminus J) ) + q(J) \} 
\\
 &= 
 \max_{K \subseteq Y_{0}} \{  
 f(  I \cup C \cup K ) - q(K) \} + q(Y_{0}) 
\\ & =
g(p^{(2)}) - M (|I|+|C|) + q(Y_{0}). 
\end{align*}
The function $g$ is submodular by 
\cite[Theorem~6.19]{Mdcasiam=valmat}
and therefore
\begin{align}
&  g_{1}(q) + g_{2}(-q)
\notag \\
& = 
g(p^{(1)}) +g(p^{(2)}) - M (|X|+|C|) + q(Y_{0})
\notag \\
 &\geq 
g(p^{(1)} \vee p^{(2)}) +g(p^{(1)} \wedge p^{(2)})
 - M (|X|+|C|) + q(Y_{0}). 
\label{g1g2gg}
\end{align}
Since
\begin{align*} 
& (p^{(1)} \vee p^{(2)})_{i}   = (p^{(1)} \wedge p^{(2)})_{i}  =
   \left\{  \begin{array}{ll}
    q_{i}          &   (i  \in Y_{0}) ,     \\
   - M     &   (i \in C ),  \\
   + M     &   (i \in N \setminus (X \cup Y)  ) , \\
                     \end{array}  \right.
\\
& (p^{(1)} \vee p^{(2)})_{i}   = - (p^{(1)} \wedge p^{(2)})_{i}  =
   + M    \quad   (i \in X_{0}),  \\
\end{align*}
we have
\begin{align} 
g(p^{(1)} \vee p^{(2)}) & \geq
f(Y) - q(Y_{0}) + M|C|,
\label{gpveep}
\\
g(p^{(1)} \wedge p^{(2)}) & \geq
f(X) + M |X| ,
\label{gpwedgep}
\end{align}
where (\ref{gpveep}) and (\ref{gpwedgep}) follow from (\ref{gpdef}) with 
$Z=Y$ and $Z=X$, respectively.
The substitution of (\ref{gpveep}) and (\ref{gpwedgep}) 
into (\ref{g1g2gg}) yields the desired inequality
$g_{1}(q) + g_{2}(-q) \geq f(X) + f(Y)$.
\end{proof}

Lemma \ref{LMmnatexcm0} implies the multiple exchange property $\McavmS$
for M-concave functions.

\begin{lemmaM}  \label{LMmexc0m}
$\McavmS$ holds for every M-concave function.
\end{lemmaM}
\begin{proof}
Let $f$ be an M-concave function.
By Proposition~\ref{PRmcav=mnatcav+equicard},
$f$ is an \Mnat-convex function such that
$|X| = |Y|$ for all $X, Y \in \dom f$.
Hence 
$\MncavmS$ for $f$ is equivalent to $\McavmS$,
whereas $f$ satisfies $\MncavmS$ by Lemma \ref{LMmnatexcm0}.
\end{proof}

Finally we establish the stronger multiple exchange property $\MncavmsS$
for \Mnat-concave functions.

\begin{lemmaM}  \label{LMmnatexc0ms}
$\MncavS$ implies $\MncavmsS$.
\end{lemmaM}
\begin{proof}
Let $f: 2^{N} \to \Rminf$ be 
an M$^{\natural}$-concave function,
and consider the associated M-concave function 
$\tilde{f}: 2^{\tilde N} \to \Rminf$
as defined in (\ref{assocMdef}),
where $\tilde{N} = N \cup S$
(cf., Proposition~\ref{PRmnatequicardvalmat}).
We apply Lemma~\ref{LMmexc0m}
to this M-concave function $\tilde{f}$. 
Suppose that we are given $X, Y \in \dom f$ and a subset $I \subseteq X \setminus Y$.
Take any $U, V \subseteq S$ with $|U|=r - |X|$ and $|V|=r - |Y|$,
where $r = \max \{ |Z| \mid Z \in \dom f \}$.
Then
$X \cup U,  Y \cup V  \in \dom \tilde f$
and
$I \subseteq (X \cup U) \setminus (Y \cup V)$.
By 
Lemma \ref{LMmexc0m}
for $\tilde f$, 
there exists 
$J \subseteq Y \setminus X$
and $W \subseteq V \setminus U$  such that
$|J| + |W| = |I|$ and
\begin{align*}
\tilde f( X \cup U) +\tilde  f( Y \cup V )  
\leq  
  \tilde f( \, ( (X \setminus I) \cup J )  \cup (U  \cup W )  \,  ) +
\tilde f( \,  ( (Y \setminus J) \cup I) \cup (V  \setminus W ) \,  ),
\end{align*}
which implies
$f( X) + f( Y )   \leq  f((X \setminus I) \cup J) +f((Y \setminus J) \cup I)$
in (\ref{mnatconcavexcmult}).
Here we have $|J| \leq |I|$ since $|J| + |W| = |I|$.
\end{proof}

\section*{Acknowledgement}

The author thanks Akiyoshi Shioura
for discussion and comments.
This work was supported by 
JSPS KAKENHI Grant Number JP20K11697.



\end{document}